\documentclass[11pt, reqno,a4paper]{article}
\usepackage{amsfonts,anysize,hyperref,amsmath,indentfirst,geometry}
\usepackage{bm}
\usepackage{amssymb}
\usepackage{hyperref}
\usepackage{mathrsfs}
\marginsize{25mm}{25mm}{15mm}{15mm}

\allowdisplaybreaks
\numberwithin{equation}{section}
\usepackage{titlesec}
\titlelabel{\thetitle.\quad}
\usepackage{graphicx}
\usepackage{hyperref}
\usepackage[symbol]{footmisc}
\usepackage{pbsi}
\usepackage{amssymb}
\usepackage{amsmath, amsthm}
\usepackage{xcolor}
\usepackage{amsfonts,amssymb,amsbsy}
\usepackage{eqnarray}
\usepackage{diagbox}
\usepackage{mathptmx}
\usepackage{mathdots}
\usepackage{algorithm}
\usepackage{algorithmic}
\newtheorem{theorem}{Theorem:}[section]
\newtheorem{definition}[theorem]{Definition}
\newtheorem{corollary}[theorem]{Corollary}
\newtheorem{lemma}[theorem]{Lemma}
\newtheorem{remark}[theorem]{Remark}
\newtheorem{example}[theorem]{Example}

\allowdisplaybreaks
\begin{document}
\title{\Large \bfseries {\color{black}{Rigorous}} Perturbation Bounds for the ${\color {black}{QX}}$ Decomposition for Centrosymmetric Matrices}
\author{Aamir Farooq$^{1,2,*}$, Rewayat Khan$^{3} $, Uzma Rani$^{4}$, M. Tariq Rahim$^{4}$}
\date{\small {$^{1}$Department of Mathematics, Zhejiang Normal University, Jinhua 321004, PR China; (aamirf88@yahoo.com)\\
$^{2}$Faculty of Mechatronics, Informatics and Interdisciplinary Studies, Institute of Information Technologies and Electronics, Technical University of Liberec, Liberec, Czech Republic\\
$^{3}$Faculty of Engineering and Natural Sciences, Sabanci University, ˙Istanbul, Turkey; (rewayat.khan@sabanciuniv.edu) \\
$^{4}$Department of Mathematics, Abbottabad University of Science and Technology, Abbottabad, Pakistan; (raniuzma205@gmail.com, tariqsms@gmail.com)}}  \maketitle
{\raggedleft\bfseries{\normalsize Abstract}}\\

\footnotetext{$^{*}$Corresponding author. Email address: \href{}{aamirf88@yahoo.com and aamirf@zjnu.edu.cn}}

Konrad Burnik suggests a structure-preserving ${\color {black}{QR}}$ factorization for centrosymmetric matrices, known as ${\color {black}{QX}}$ factorization. In this article, {\color{black} we obtain the} explicit expressions for {\color{black}{rigorous}} perturbation bounds of the ${\color {black}{QX}}$ factorization when the original matrix is perturbed, either norm-wise or component-wise. First, using the matrix-equation approach, weak {\color{black}{rigorous}} perturbation bounds are derived. Then, strong {\color{black}{rigorous}} perturbation bounds are obtained by combining the modified matrix-vector equation approach with the strategy for the Lyapunov majorant function and the Banach fixed-point theorem. The mixed and component-wise condition numbers and their upper bounds are also explicitly expressed. Numerical tests illustrate the validity of the obtained results.
\vspace{.3cm}

{\raggedleft \em Keywords:} Structure-preserving ${\color {black}{QR}}$ factorization; Centrosymmetric matrix; {\color{black} {Rigorous}} perturbation
bounds; Mixed and Component-wise condition numbers; Lyapunov majorant function; Banach fixed point
theorem.

\section{Introduction}\label{sec.1}
\vspace{-2pt}
\vspace{-2pt}
Let ${\mathbb {R}}^{m \times n}$ be the set of $m \times n$ real matrices, ${\mathbb {R}}_{r}^{m \times n}$ the set of rank-$r$ matrices in ${\mathbb {R}}^{m \times n}$, and $I_{r}$ the identity matrix of order $r$. A centrosymmetric matrix \( A = (a_{\alpha,\beta}) \in \mathbb{R}^{m\times n} \) is such that \( a_{\alpha{\color {blue}{,}}\beta} = a_{{m-\alpha+1},{n-\beta+1}} \), where \( \alpha = 1:m \) and \( \beta = 1:n \), i.e., $ {\color{black}{R_{m}}}A{\color{black}{R_{n}}} = A $, where $ {\color{black}{R_{n}}} $ is the reverse identity matrix of order \( n \), given by

\begin{align*}
  {\color{black}{R_{n}}} = \left[
    \begin{array}{ccccc}
       &  &  &  & 1 \\
       &  & \iddots &   \\
      1 &  &  &  &
    \end{array}
  \right].
\end{align*}

Every centrosymmetric matrix can be factorized as
\begin{align}\label{1.1}
A ={\color{black}{Q_{0}}}{\color{black}{X_{0}}},
\end{align}
where ${\color{black}{Q_{0} = ({q^{0}} }}\! \!_{\alpha,\beta}) \in \mathbb {R}^{m \times m}$ is a perplectic orthogonal matrix, satisfying ${\color{black}{Q_{0}}}^{T} {\color{black}{R_{m}}}{\color{black}{Q_{0}}}= {\color{black}{R_{m}}}$ and $ {\color{black}{Q_{0}}}^{T} {\color{black}{Q_{0}}} = I_{m}, $ and ${\color{black}{X_{0}}} = ({x^{0}}\! \!_{\alpha,\beta}) \in\mathbb {R}^{ m \times n}$  is a double-cone centrosymmetric matrix (see \cite{burnik2015structure}). The name given to this decomposition is structure-preserving ${\color {black}{QR}}$ \cite{burnik2015structure}, or ${\color {black}{QX}}$ decomposition \cite{steele2018qx}. Burnik \cite{burnik2015structure} demonstrated this decomposition by employing a Householder-type algorithm based on perplectic orthogonal block reflectors. Steele et al. \cite{steele2018qx} also validated this by utilizing standard $QR$ factorizations on two matrices, each approximately half the size of the matrix $A$. The thin ${\color {black}{QX}}$ decomposition is also discussed in \cite{steele2018qx}:

\begin{align}\label{1.2}
    A = {\color{black}{QX}},
\end{align}
where ${\color{black} {Q}} \in \mathbb{R}^{m \times n}$ is a column perplectic orthogonal matrix defined by ${\color{black} {Q}}^{T}{\color{black}{R_{m}}}{\color{black} {Q}} = {\color{black}{R_{n}}}$ and ${\color{black} {Q}}^{T}{\color{black} {Q}} = I_{n}$. The matrix ${\color{black} {X}} = (x_{\alpha,\beta}) \in \mathbb{R}^{n \times n}$ is a square double-cone centrosymmetric matrix, which we refer to as an "\textit{${\color{black} {X}}$}"-type matrix \cite{lv2022perturbation}. It is worth noting that both ${\color{black} {Q}}_{0}$ and ${\color{black} {Q}}$ are centrosymmetric matrices. When $n$ is even, one can obtain the factorization \eqref{1.2} from \eqref{1.1} by eliminating the middle $m-n$ zero rows of ${\color{black} {X}}_{0}$ and the corresponding middle columns of ${\color{black} {Q}}_{0}$. However, when $n$ is odd, an additional transformation is required. For insights into the uniqueness of this transformation, and its implementation, refer to \cite{lv2022perturbation}. For simplicity, we always refer to the thin ${\color {black}{QX}}$decomposition \eqref{1.2} when discussing the ${\color {black}{QX}}$decomposition below. \\

Centrosymmetric matrices hold significant importance across diverse fields such as finite Markov processes, signal interpolation, pattern recognition, wavelet analysis, and a broad spectrum of applications in physics, engineering, and digital processing \cite{iosifescu2014finite,daubechies1992ten,datta1986some,hanna2003centrosymmetric,datta1989reducibility,li1999computationally}. The ${\color {black}{QX}}$ decomposition for centrosymmetric matrices is applicable for solving centrosymmetric systems of linear equations \cite{burnik2015structure}, as well as for designing practical structure-preserving algorithms. For example, a structure-preserving ${\color {black}{QX}}$ algorithm can be developed using the ${\color {black}{QX}}$ decomposition. The method presented in \cite{mackey2005structure} is a structure-preserving ${\color {black}{QX}}$algorithm designed to solve the eigenvalue problem of symmetric persymmetric matrices. Regarding centrosymmetric matrices, some scholars have studied their basic properties, inverse and pseudoinverse, eigenvalue problems, and various inverse problems (see, e.g., \cite{andrew1998classroom, weaver1985centrosymmetric, li2011numerical, bai2005inverse, zhou2003least}, and the references therein). Although computations involved in the ${\color {black}{QX}}$ factorization are generally sensitive, perturbation bounds and condition numbers can be utilized to estimate the accuracy of simulations. Recently, \cite{lv2022perturbation} presented a perturbation analysis for the ${\color {black}{QX}}$ decomposition of centrosymmetric matrices, deriving only the first-order norm-wise perturbation bounds for this decomposition.\\
\vspace{-2pt}

{\color{black} {Rigorous}} perturbation bounds are studied because first-order bounds may be inaccurate as they neglect higher-order terms. We continue to study {\color{black} {rigorous}} perturbation bounds, and the mixed and component-wise condition numbers for this decomposition. This paper will primarily focus on deriving {\color{black} {rigorous}} perturbation bounds and analyzing the mixed and component-wise condition numbers. Condition numbers play an important role in estimating the forward errors of algorithms when the original system becomes ill-conditioned. They are also used to predict the convergence of iterative methods. For the ${\color {black}{QX}}$ factorization, condition numbers must be investigated to evaluate the worst-case sensitivity of factors ${\color{black} {Q}}$ and ${\color{black} {X}}$ when the original matrix $A$ is perturbed.\\

The structure of this paper is as follows: In Section \ref{sec.2}, we explore {\color{black} rigorous} perturbation bounds, focusing on both norm-wise and component-wise perturbations for the ${\color {black}{QX}}$ factorization of centrosymmetric matrices. Section \ref{sec.3} provides explicit expressions for two types of condition numbers: mixed, and component-wise. These condition numbers are examined because norm-wise condition numbers often overestimate the problem by failing to capture the sparsity, or scaling structures, of coefficient matrices. Section \ref{sec.4} presents numerical experiments that validate our theoretical findings. Before proceeding to subsequent sections, we outline relevant notations and preliminaries, concluding in Section \ref{sec.5}. Most notations and preliminaries are based on references \cite{chang1997pertubation, li2016new, li2017note, li2015improved}.
\\
\vspace{-2pt}

For a given matrix $\mathrm{A} = (a_{\alpha\beta}) \in \mathbb{R}^{m \times n}$, the condition number of $\mathrm{A}$ is defined as $\kappa_{2} = \|\mathrm{A}\|_{2}\|\mathrm{A}^{\dag}\|_{2}$. Here, $\left\| \mathrm{A} \right\|_2$ and $\left\|\mathrm{A} \right\|_F$ denote the spectral norm and Frobenius norm of $\mathrm{A}$, respectively. For these two matrix norms, the following inequalities hold (see \cite[pp. 80]{chang1997pertubation}):
\begin{eqnarray}\label{1.3}
\left\|\mathrm{UVW}\right\|_2 \leq \left\| \mathrm{U} \right\|_2\left\| \mathrm{V} \right\|_2\left\| \mathrm{W} \right\|_2,\quad \left\| \mathrm{UVW} \right\|_F \leq \left\| \mathrm{U} \right\|_2 \left\| \mathrm{V} \right\|_F \left\| \mathrm{W} \right\|_2,
\end{eqnarray}
which are valid only when the matrices $\mathrm{U,V,}$ and $\mathrm{W}$ are conformable for multiplication.
For a given matrix, $C=[\textsl{c}_{1},\textsl{c}_{2},...,\textsl{c}_{p}]=(\textsl{c}_{\alpha\beta})\in\mathbb{R}^{p\times p}$, the vector of the first $\alpha$ elements of $\textsl{c}_{\beta}$ is denoted by $c_{\beta}^{^{(\alpha)}}$. Here we use the same operators as in \cite{chang1997pertubation}, i.e.,
\begin{eqnarray*}{\rm {\color{black}{up}}}\left(C \right) := \left[\begin{array}{cccc}
\frac{1}{2}\textsl{c}_{11} & \textsl{c}_{12} & \cdots & \textsl{c}_{1p}\\
0 & \frac{1}{2}\textsl{c}_{22} & \cdots & \textsl{c}_{2p}\\
\vdots & \vdots & \ddots & \vdots \\
0 & 0 & \cdots & \frac{1}{2}\textsl{c}_{pp}
\end{array} \right],~~
{\rm {\color{black}{ut}}}\left(C \right) := \left[ {\begin{array}{*{20}{c}}
{{\textsl{c}_{11}}}&{{\textsl{c}_{12}}}& \cdots &{{\textsl{c}_{1p}}}\\
0&{{\textsl{c}_{22}}}& \cdots &{{\textsl{c}_{2p}}}\\
\vdots&\vdots& \ddots & \vdots \\
0&0&\cdots&{{\textsl{c}_{pp}}}
\end{array}} \right],
\end{eqnarray*}
\begin{eqnarray*}
{\rm {\color{black}{vecu}}}\left( C \right) := \left[
                                     \begin{array}{c}
                                       \textsl{c}_{1}^{(1)} \\
                                       \textsl{c}_{2}^{(2)}\\
                                       \vdots \\
                                       \textsl{c}_{p}^{(p)} \\
                                     \end{array}\right] \in \mathbb{R}^{(p(p+1))/2}, \quad
{\rm {\color{black}{vec}}}\left( C \right):=\left[\begin{array}{c}
                                       \textsl{c}_{1} \\
                                        \textsl{c}_{2} \\
                                       \vdots \\
                                      \textsl{ c}_{p} \\
\end{array}\right] \in \mathbb{R}^{p^2},  \quad {\rm {\color{black}{low}}}(C)= C - {\rm {\color{black}{up}}}(C)={\rm {\color{black}{up}}}(C^{T})^{T},
\end{eqnarray*}
where ${\rm {\color{black}{vec}}}(C)$ is defined as the vector obtained by stacking all columns of the matrix $C$ into a vector. For any centrosymmetric matrix $ C= (\lambda_{\alpha,\beta}) \in \mathbb {R} ^{n \times n}$ and $ n=2l$, we can write it in the following form
\begin{align*}
& C = [\lambda_{1}, \lambda_{2}, . . . , \lambda_{l}, R_{n}\lambda_{l},. . .,  R_{n}\lambda_{2}, R_{n}\lambda_{1}]\\
&=\left[ \begin{array}{cccccccc}
         \lambda_{1,1}&\lambda_{1,2}&\cdots& \lambda_{1,l}& \lambda_{n,l} & \cdots& \lambda_{n,2} & \lambda_{n,1} \\
           \lambda_{2,1}&\lambda_{2,2}&\cdots& \lambda_{2,l}& \lambda_{n-1,l} & \cdots& \lambda_{n-1,2} & \lambda_{n-1,1} \\
           \vdots &\vdots &  &\vdots &\vdots &  &\vdots &\vdots \\
           \lambda_{l,1}&\lambda_{l,2}&\cdots& \lambda_{l,l}& \lambda_{l+1,l} & \cdots& \lambda_{l+1,2} & \lambda_{l+1,1} \\
           \lambda_{l+1,1}&\lambda_{l+1,2}&\cdots& \lambda_{l+1,l}&\lambda_{l,l} & \cdots& \lambda_{l,2} & \lambda_{l,1} \\
           \vdots &\vdots &  &\vdots &\vdots &  &\vdots &\vdots \\
           \lambda_{1,1}&\lambda_{1,2}&\cdots& \lambda_{1,l}& \lambda_{n,l} & \cdots& \lambda_{2,2} & \lambda_{2,1} \\
           \lambda_{1,1}&\lambda_{1,2}&\cdots& \lambda_{1,l}& \lambda_{n,l} & \cdots& \lambda_{1,2} & \lambda_{1,1}
         \end{array}
\right].
\end{align*}
Here $\lambda_{\beta} ^{\alpha}$ denotes the vector formed by consecutively assembling the leading $\alpha$ and the last $\alpha$ elements
of $\lambda_{\beta}^{(\alpha)}$. Additionally, we define
\begin{align*}
{\rm {\color{black}{upx}}}\left(C \right) :=\left[ \begin{array}{cccccccc}
         \frac{1}{2}\lambda_{1,1}&\lambda_{1,2}&\cdots& \lambda_{1,l}& \lambda_{n,l} & \cdots& \lambda_{n,2} & \frac{1}{2}\lambda_{n,l} \\
           &\frac{1}{2}\lambda_{2,2}&\cdots& \lambda_{2,l}& \lambda_{n-1,l} & \cdots& \frac{1}{2}\lambda_{n-1,2} & \\
            & &\ddots  &\vdots &\vdots &\iddots  & & \\
           & & &\frac{1}{2} \lambda_{l,l}& \frac{1}{2}\lambda_{l+1,l} & &  &  \\
          & & &\frac{1}{2} \lambda_{l+1,l}&\frac{1}{2} \lambda_{l,l} & &  &  \\
            & & \iddots &\vdots &\vdots &\ddots & & \\
           &\frac{1}{2}\lambda_{1,2}&\cdots& \lambda_{1,l}& \lambda_{n,l} & \cdots& \frac{1}{2}\lambda_{2,2} &  \\
           \frac{1}{2}\lambda_{1,1}&\lambda_{1,2}&\cdots& \lambda_{1,l}& \lambda_{n,l} & \cdots& \lambda_{1,2} &\frac{1}{2} \lambda_{1,1}
         \end{array}\right],
\end{align*}
\begin{align*}
{\rm {\color{black}{utx}}}\left( C \right) :=\left[ \begin{array}{cccccccc}
         \lambda_{1,1}&\lambda_{1,2}&\cdots& \lambda_{1,l}& \lambda{n,l} & \cdots&\lambda_{n,2} &\lambda_{n,1} \\
           &\lambda_{2,2}&\cdots&\lambda_{2,l}& \lambda_{n-1,l} & \cdots&\lambda_{n-1,2} & \\
            & &\ddots  &\vdots &\vdots &\iddots  & & \\
           & & &\lambda_{l,l}& \lambda_{l+1,l} & &  &  \\
          & & & \lambda_{l+1,l}&\lambda_{l,l} & &  &  \\
            & & \iddots &\vdots &\vdots &\ddots & & \\
           &\lambda_{1,2}&\cdots& \lambda_{1,l}&\lambda_{n,l} & \cdots& \lambda_{2,2} &  \\
           \lambda_{1,1}&\lambda_{1,2}&\cdots& \lambda_{1,l}& \lambda_{n,l} & \cdots& \lambda_{1,2} &\lambda_{1,1}
         \end{array}
\right],
\end{align*}
\begin{align*}
{{\color{black} {X}}} =\left[ \begin{array}{c}
                   \lambda_{1}^{(1)} \\
                  \lambda_{2}^{(2)} \\
                   \vdots\\
                  \lambda_{l}^{(l)}
                 \end{array}
                 \right], \,\,\,\,{\color{black}{ \rm xvec}}(\textit{C})= \left[
                              \begin{array}{c}
                                y \\
                                R_{l(l+1)}y \\
                              \end{array}
                            \right]\in \mathbb{R}^{\tau_{1}},
 \end{align*}
 and ${\rm {\color{black}{lowx}}}(C)= C - {\rm{\color{black}{upx}}}(C)={\rm {\color{black}{upx}}}(C^{T})^{T}$, where $\tau_{1} = n(n + 2)/2$. Note that the symbol ${\color{black} {X}}$ indicates the main diagonal and off diagonal; for instance, ${\rm {\color{black}{utx}}}(B)$ indicates the centrosymmetric matrix obtained by the upper triangular part of \textit{C} related to the main diagonal and off diagonal. For the construction of these operators, we have
 \begin{align}
&\quad {\rm {\color{black}{ \rm xvec}}}(C) = {M_{{\rm {{\color{black}{ \rm xvec}}}}}}{\rm {\color{black}{vec}}}(C),\quad
{\rm {\color{black}{vec}}}\left({{\rm {\color{black}{upx}}}\left( C \right)} \right) = {M_{{{\color{black}{\rm upx}}}}}{\rm {\color{black}{vec}}}(C),\quad {\rm {\color{black}{vec}}}\left( {{\rm {\color{black}{utx}}}\left( C \right)} \right) = {M_{{\rm {\color{black}{utx}}}}}{\rm {\color{black}{vec}}}(C),\label{1.4}
\end{align}
where
\begin{align*}
&{M_{{\rm {\color{black}{xvec}}}}} = {\rm diag}\left( M, {R}_{\tau_1} M {R}_{n^2/2} \right) \in \mathbb{R}^{\tau_1 \times n^2}, \quad \text{with} \\
&{M} = {\rm diag}\left( P_1, P_2, \dots, P_l \right) \in \mathbb{R}^{\tau_1/2 \times n^2/2}, \quad \text{and} \quad {P_\alpha} = \left[ I_{\alpha \times (n-\alpha)}, I_{\alpha} \right] \in \mathbb{R}^{2\alpha \times n},\\
&{M_{{\rm {\color{black}{utx}}}}} = {\rm diag}\left( \hat{M}, {R}_{n^2/2}\hat{M} {R}_{n^2/2} \right) \in \mathbb{R}^{n^2 \times n^2}, \quad \text{with} \\
&\hat{M} = {\rm diag}\left( \hat{P}_1, \hat{P}_2, \dots, \hat{P}_l \right) \in \mathbb{R}^{n^2/2 \times n^2/2}, \quad \text{and} \quad \hat{P}_\alpha = {\rm diag}\left( I_{\alpha-1}, \frac{1}{2}, 0_{n-2\alpha}, \frac{1}{2}, I_{\alpha-1} \right) \in \mathbb{R}^{n \times n},\\
&{M_{{\rm {\color{black}{upx}}}}} = {\rm diag}\left( \tilde{M}, {R}_{n^2/2}\tilde{M} {R}_{n^2/2} \right) \in \mathbb{R}^{n^2 \times n^2}, \quad \text{with} \\
&\tilde{M} = {\rm diag}\left( \tilde{P}_1, \tilde{P}_2, \dots, \tilde{P}_l \right) \in \mathbb{R}^{n^2/2 \times n^2/2}, \quad \text{and} \quad \tilde{P}_\alpha = {\rm diag}\left( I_{\alpha-1}, 0_{n-2\alpha}, I_\alpha \right) \in \mathbb{R}^{n \times n}{\color{blue}{,}}
\end{align*}
where $0_{l\times m}$  and $0_{l}$ denote the zero matrix of order $l\times m$ and $l$, respectively, and ${\rm I_{\textit{l} \times m} = [ I_{\textit{l}} \quad 0_{\textit{l}\times(m-\textit{l})} ], }$ when $m \geq l$ . Also, we have
\begin{equation}\label{1.5}
  \|{\rm {\color{black}{upx}}}(C)\|_{F}\leq \|C\|_{F}.
\end{equation}
If $C=C^{T}$, then we have from \cite{chang1997pertubation},
\begin{equation}\label{1.6}
   \|{\rm {\color{black}{upx}}}(C)\|_{F}\leq \frac{1}{\sqrt{2}}\|C\|_{F}.
\end{equation}
Additionally,
\begin{equation}\label{1.7}
   \|{\rm {\color{black}{upx}}}(C+C^{T})\|_{F}\leq {\sqrt{2}}\|C\|_{F},
\end{equation}
\begin{eqnarray}\label{1.8}
{M_{{\rm {\color{black}{ \rm xvec}}}}}M_{{\rm {\color{black}{ \rm xvec}}}}^T = {I_{{\tau_1}}},\quad M_{{\rm {\color{black}{ \rm xvec}}}}^T{M_{{\rm {\color{black}{ \rm xvec}}}}} = {M_{{\rm {\color{black}{utx}}}}}.
\end{eqnarray}
Consider that ${\rm {\color{black}{uvec}}}^{\dag}:\mathbb{R}^{\tau_1} \to \mathbb{R}^{n \times n}$ is a right inverse of the operator ${\rm {\color{black}{uvec}}}$ such that ${\rm {\color{black}{uvec}}} \cdot {\rm {\color{black}{uvec}}}^{\dag} = {1_{{\tau_1} \times {\tau_1}}}$ and ${\rm {\color{black}{uvec}}}^{\dag } \cdot {\rm {\color{black}{uvec}}} = {\rm {\color{black}{ut}}}$. Then, the matrix of the operator ${\rm {\color{black}{uvec}}}$ is $M_{{\rm {\color{black}{uvec}}}}^T$, i.e., ${\rm {\color{black}{uvec}}}^{\dag}(A) = M_{{\rm {\color{black}{uvec}}}}^T {\rm {\color{black}{vec}}}(A)$.

\begin{lemma}\label{lem1}
\cite{lv2022perturbation}: Let $\mathbb{D}_n$$ \subseteq \mathbb{R}^{n\times n}$ be the set of diagonal centrosymmetric matrices with positive diagonal elements and  $n=2l$. Then, for any $D_n=\textmd{diag}(\delta_{1},\delta_{2},...,\delta_{n},\delta_{n},...,\delta_{2},\delta_{1})\in ${$\mathbb{D}_n$} and centrosymmetric matrix $ A = (a_{\alpha,\beta}) \in\mathbb{R}^{n\times n},$
we have
\begin{center}
\begin{align}
{\rm {\color{black}{upx}}}(AD_{n})={\rm {\color{black}{upx}}}(A)D_{n},\quad D_{n}{\rm {\color{black}{up}}x}(A)={\rm {\color{black}{upx}}}(D_{n}A), \label{1.9}\\
{\rm {\color{black}{lowx}}}(AD_{n})={\rm {\color{black}{lowx}}}(A)D_{n},\quad D_{n}{\rm {\color{black}{lowx}}}(A)={\rm {\color{black}{lowx}}}(D_{n}A),\label{1.10}\\
{\left\| {{\rm {\color{black}{upx}}}\left( A \right) + {D_{n}}^{ - 1}{\rm {\color{black}{upx}}}\left( {{A^T}} \right)D_{n}} \right\|_F} \le \sqrt {1 + \varsigma_{D_{n}}^{2}} {\left\| A \right\|_F},\label{1.11}\\
{\left\| { D_{n}{\rm {\color{black}{lowx}}}\left( A \right) - {D_{n}}^{ - 1}{\rm {\color{black}{lowx}}}\left({{A^T}} \right)D_{n}} \right\|_F} \le \sqrt {2}
\varsigma_{D_{n}} {\left\| A \right\|_F},\label{1.12}
\end{align}
\end{center}
where $ \varsigma_{D_{n}} = \mathop {\max }\limits_{1 \le \alpha < \beta \le n} \{\delta_{\beta}/\delta_{\alpha}\}$.
\end{lemma}

The Kronecker product is another useful tool for obtaining explicit expressions of condition numbers. Let $E=(\textsl{e}_{\alpha\beta})\in \mathbb{R}^{r\times s}$ and $F\in \mathbb{R}^{p\times q}$. The Kronecker product of two matrices $E$ and $F$ is defined by (see,  for instance, \cite{chang1997pertubation}):
\begin{equation*}
  E\otimes F = \left[ {\begin{array}{*{20}{c}}
{ {\textsl{e}_{11}}F}&{{\textsl{e}_{12}}F}& \cdots &{ {\textsl{e}_{1n}}F}\\
{ {\textsl{e}_{21}}F}&{ {\textsl{e}_{22}}F}& \cdots &{ {\textsl{e}_{2n}}F}\\
 \vdots & \vdots & \ddots & \vdots \\
{ {\textsl{e}_{m1}}F}&{{\textsl{e}_{m2}F}}& \cdots &{{\textsl{e}_{mn}}F}
\end{array}} \right],
\end{equation*}
 whose properties are listed as follows:
\begin{align}
&{\rm {\color{black}{vec}}}\left( {ECF} \right) = ( {{F^T}\otimes E} ){\rm {\color{black}{vec}}}\left( C \right),\label{1.13}\\
&{\left\| {F\otimes E} \right\|_2} = {\left\| F \right\|_2}{\left\| E \right\|_2},\label{1.14}\\
&( {F\otimes E} )( {H\otimes G} ) = ( {FH\otimes EG} ),\label{1.15}\\
&{( {F\otimes E} )^{ - 1}} = {F^{ - 1}}\otimes {E^{ - 1}},\ \text{if}\  F\  \text{and}\  E\  \text{are nonsingular}, \label{1.16}\\
&\Pi_{m,l} {\rm {\color{black}{vec}}}(E) = {\rm {\color{black}{vec}}}(E^{T}),\label{1.17}
\end{align}
where $\Pi$ is the \( \rm {\color{black}{vec}} \)-permutation matrix, and the matrices $H$ and $G$ are of appropriate orders. We define the entry-wise division between two vectors $x$ and $y$ as $\frac{x}{y}=[z_{1},...,z_{p}]^T$ with $x=[x_{1},\cdots,x_{p}]^T,$ $y=[y_{1},\cdots,y_{p}]^T\in\mathbb{R}^{p}$ such that
\begin{align*}
z_{i}=\begin{cases}\frac{x_{i}}{y_{i}},\,\,\  &{\rm if}\,\ y_{i}\neq 0,\\
      x_{i},\,\,\                             &{\rm if}\,\ y_{i}= 0. \end{cases}
\end{align*}
Following \cite{xie2013condition}, the component-wise distance between $x$ and $y$ is defined as
\begin{align*}
d(x,y)=\left\|\frac{x-y}{y}\right\|_\infty=\mathop{\max}\limits_{i=1,\cdots,p}\left\{\frac{|x_{i}-y_{i}|}{|y_{i}|}\right\}=\begin{cases}\frac{|x_{i_{0}}-y_{i_{0}}|}{|y_{i_{0}}|},\,\,\  &{\rm if}\,\ y_{i_{0}}\neq 0,\\
      |x_{i_{0}}|,\,\,\                             &{\rm if}\,\ y_{i_{0}}= 0. \end{cases}
\end{align*}
Note that when $y_{i_{0}}\neq 0,$ $d(x,y)$ will give the relative distance from $x$ to $y$ with respect to $y$, while the absolute distance is used when $y_{i_{0}}= 0$.

In order to define the mixed and componentwise condition numbers, we also need to introduce the
set $B^{0}(x,\epsilon) =\{{a = [a_{1}, a_{2},\cdots, a_{p}]^{T} \in \mathbb{R}^{p}||a_{i}-x_{i}|\leq \epsilon |x_{i}|, i=1,\cdots,p}\}$, with $x = [x_{1}, x_{2}, \cdots, x_{p}]^{T} \in \mathbb{R}^{p} $ and $\epsilon >0,$ and we denote by Dom($F$) the domain of definition of the function $F:\mathbb{R}^{p} \rightarrow \mathbb{R}^{q}.$ Therefore, the definitions of the mixed and component-wise condition numbers are as follows.

\begin{definition}\label{dfn2.2}
{\rm \cite{cucker2007mixed}:} Let $F:\mathbb{R}^{p}\rightarrow \mathbb{R}^{q}$ be a continuous function defined on an open set Dom($F$)$\subset \mathbb{R}^{p},$ and $x\in {\rm Dom}{(F)},\ x\neq 0$ such that $F(x)\neq 0.$ Then, \\
(\romannumeral1) the mixed condition number for $F$ at $x$ is defined as
\begin{equation*}
m(F,x)=\lim_{\varepsilon\rightarrow 0}\sup_{\substack{a\in B^{0}(x,\epsilon) \\a\neq x}}\frac{\|F(a)-F(x)\|_{\infty}}{\|F(x)\|_{\infty}}\frac{1}{d(a,x)},
\end{equation*}
(\romannumeral2) the component-wise condition number for $F$ at $x$ is defined  as
\begin{equation*}
  c(F,x)=\lim_{\varepsilon\rightarrow 0}\sup_{\substack{a\in B^{0}(x,\epsilon)\\a\neq x}}\frac{d\big(F(a),F(x)\big)}{d(a,x)}.
\end{equation*}
\end{definition}

The following lemma from {\cite{cucker2007mixed}} simplifies the condition number computation when the map $F$ in Definition \ref{dfn2.2} is Fr\'{e}chet differentiable.

\begin{lemma}\label{lem2.1}
Using the same assumptions as in Definition \ref{dfn2.2}, let us assume that $F$ is Fr{\'e}chet differentiable at $x$. Then, we have
\begin{center}
$m(F, x)= \frac{\| |DF(x)||x|\|_{\infty}}{\|F(x)\|_{\infty}}$,\\
$c(F, x)= \left\|\frac{|DF(x)||x|}{|F(x)|}\right\|_{\infty}$,
\end{center}
where $DF(x)$ is the Fr\`{e}chet derivative of $F$ at $x$.
\end{lemma}

\vspace{-6pt}

\section{Rigorous Perturbation Bounds} \label{sec.2}
\vspace{-2pt}

In this section, we employ two distinct methods to derive rigorous perturbation bounds with norm-wise perturbation for the ${\color{black}{QX}}$-factorization: the refined matrix equation approach and the modified matrix-vector equation technique. Additionally, we establish rigorous perturbation bounds for the ${\color{black}{QX}}$-factorization under component-wise perturbations. The results are presented in the following three subsections.

\subsection{Rigorous Perturbation Bounds By Using Refined Matrix Equation Method}
To derive weak rigorous perturbation bounds for the factors, ${\color{black} {X}}$ and ${\color{black} {Q}}$ in Theorem \ref{thm2.1}, we utilized the refined matrix equation approach.
\begin{theorem} \label{thm2.1}
{Assume that \( A \in \mathbb{R}^{m \times n} \) has a ${\color{black}{QX}}$ decomposition as in \eqref{1.2}, and let \( \Delta A \in \mathbb{R}^{m \times n} \) be centrosymmetric. If \( \|\Delta A A^{\dagger}\|_{2} < 1 \), then there exists a unique ${\color{black}{QX}}$ decomposition of \( A + \Delta A \), i.e.,
\begin{align}\label{2.1}
A + \Delta A = ( {\color{black} {Q}} + \Delta {\color{black} {Q}}) ( {\color{black} {X}}  + \Delta {\color{black} {X}}),
\end{align}
  with
  \begin{align}
& \|\Delta {\color{black} {X}}\|_{F} \leq (\sqrt{6} + \sqrt{3}) \left(\mathop {\inf}\limits_{D_{n} \in \mathbb{D}_n} \sqrt{1 + {\varsigma^{2}}_{D_{n}}}\kappa_{2} (D_{n}^{-1}{{\color{black} {X}}})\right) \|{\color{black} {Q}}\|_{2}\|\Delta A\|_{F}, \label{2.2} \\
& \|\Delta {\color{black} {Q}}\|_{F} \leq (2\sqrt{2} + 2) \left ( \mathop {\inf}\limits_{D_{n} \in \mathbb{D}_n} \|{\color{black} {Q}}{D^{-1}}_{n}\|_{2}\|{\color{black} {X}}^{-1}D_{n}\|_{2}\right) \|{\color{black} {Q}}\|_{2} \|\Delta A\|_{F}\nonumber\\
& \qquad\qquad\qquad + (2\sqrt{3} + \sqrt{6}) \|{\color{black} {Q}}^{T} \Delta A {\color{black} {X}}^{-1}\|_{F} ,  \label{2.3}
  \end{align}
  under the condition
  \begin{align}\label{2.4}
     & \|{\color{black} {Q}}^{T}\Delta A {\color{black} {X}}^{-1}\|_{F} \leq \sqrt{\frac{3}{2}}-1.
  \end{align}}
\end{theorem}
\begin{proof}
\rm{{\color{black} Noting that $A$ is a full column rank centrosymmetric matrix, we have that $A + \Delta A$ is also a full column rank centrosymmetric matrix when the condition $\|\Delta A {A}^{\dagger}\|_{2} < 1 $ holds.} Thus, $ A + \Delta A$ {\color{black} has the} ${\color {black}{QX}}$ factorization \eqref{2.1} and together with the fact that $( {\color{black} {Q}} + \Delta {\color{black} {Q}} )^{T}( {\color{black} {Q}} + \Delta {\color{black} {Q}} ) = I_n,$ we obtain
\begin{align*}
( A +\Delta A)^T ( A +\Delta A) = ( {\color{black} {X}} +\Delta {\color{black} {X}})^T ( {\color{black} {X}} +\Delta {\color{black} {X}}),
\end{align*}
utilizing \eqref{1.2} and $A^T A = {\color{black} {X}}^T {\color{black} {X}}$ lead to
\begin{align*}
{\color{black} {X}}^{T}\Delta {\color{black} {X}} + {\Delta {\color{black} {X}}^{T}}{\color{black} {X}} = {\color{black} {X}}^{T}{{\color{black} {Q}}^T}\Delta A + {\Delta A}^T{\color{black}{QX}} + {\Delta A}^T{\Delta A} - {\Delta {\color{black} {X}}}^T\Delta {\color{black} {X}}.
\end{align*}
Pre-multiplying by ${\color{black} {X}}^{-T} $ and post-multiplying by ${\color{black} {X}}^{-1},$ we have
\begin{align}\label{2.5}
  & \Delta {\color{black} {X}}{\color{black} {X}}^{-1} + {{(\Delta {\color{black} {X}}{\color{black} {X}}^{-1})}^{T}} =  {\color{black} {Q}}^T\Delta A{\color{black} {X}}^{-1} +  ({{\color{black} {Q}}^T\Delta A{\color{black} {X}}^{-1} })^{T} + {\color{black} {X}}^{-T}{\Delta A}^{T}\Delta A{\color{black} {X}}^{-1}\nonumber\\
  &\quad\quad\quad\quad\quad\quad\qquad\qquad\quad - {\color{black} {X}}^{-T}{\Delta {\color{black} {X}}}^{T}\Delta {\color{black} {X}}{\color{black} {X}}^{-1}.
 \end{align}
Note that ${\color{black} {X}}$ and $\Delta {\color{black} {X}}$ are ${\color{black} {X}}$-type matrices, then Lemma 2.3 in \cite{lv2022perturbation}
states that $\Delta {\color{black} {X}}{\color{black} {X}}^{-1}$ is also a ${\color{black} {X}}$- type matrix. Using the operator
$\rm {\color{black}{upx}}$ on the preceding equation yields:
 \begin{align}\label{2.6}
 & \Delta {\color{black} {X}}{\color{black} {X}}^{-1} = {\rm {\color{black}{upx}}} [ {\color{black} {Q}}^T\Delta A{\color{black} {X}}^{-1} +  ({{\color{black} {Q}}^T\Delta A{\color{black} {X}}^{-1} })^{T} + {\color{black} {X}}^{-T}\Delta A^{T}\Delta A{\color{black} {X}}^{-1} - {\color{black} {X}}^{-T}\Delta X^{T}\Delta {\color{black} {X}}{\color{black} {X}}^{-1}] \nonumber\\
 & \Delta {\color{black} {X}}{\color{black} {X}}^{-1} = {\rm {\color{black}{upx}}} [ {\color{black} {Q}}^T\Delta A{\color{black} {X}}^{-1} +  ({{\color{black} {Q}}^T\Delta A{\color{black} {X}}^{-1} })^{T}] + {\rm {\color{black}{upx}}}[{\color{black} {X}}^{-T}\Delta A^{T}\Delta A{\color{black} {X}}^{-1}]- {\rm {\color{black}{upx}}}[{\color{black} {X}}^{-T}\Delta {\color{black} {X}}^{T}\Delta {\color{black} {X}}{\color{black} {X}}^{-1}].
 \end{align}
Using the Frobenius norm on the preceding equation, and  by
 \eqref{1.3},{\color{black}{\eqref{1.5}, and \eqref{1.6} the quantity}} $ \|\Delta {\color{black} {X}}{\color{black} {X}}^{-1}\| $ satisfies
 \begin{align}\label{2.7}
 & {\|\Delta {\color{black} {X}}{\color{black} {X}}^{-1}\|}_{F} \leq \sqrt{2}\|{\color{black} {Q}}^T\Delta A{\color{black} {X}}^{-1}\|_{F} + \frac{1}{\sqrt{2}}\| {\color{black} {X}}^{-T}\Delta A^{T}\Delta A{\color{black} {X}}^{-1}\|_{F} + \frac{1}{\sqrt{2}}\| {\color{black} {X}}^{-T}\Delta {\color{black} {X}}^{T}\Delta {\color{black} {X}}{\color{black} {X}}^{-1}\|_{F} \nonumber\\
 & \|\Delta {\color{black} {X}}{\color{black} {X}}^{-1}\|_{F} \leq \sqrt{2}\|{\color{black} {Q}}^T\Delta A{\color{black} {X}}^{-1}\|_{F} + \frac{1}{\sqrt{2}}\left({\|{{\color{black} {Q}}}\Delta A {\color{black} {X}}^{-1}\|_{F}}^{2} + {\| \Delta {\color{black} {X}} {\color{black} {X}}^{-1}\|_{F}}^{2} \right),
 \end{align}
{ where $\|{\color{black} {X}}^{-T}\Delta A^{T}\Delta A\hat{{\color{black} {X}}^{-1}}\|_{F}= \|{\color{black} {X}}^{-T}\Delta A^{T}{\color{black} {Q}}^{T}{\color{black} {Q}}\Delta A\hat{{\color{black} {X}}^{-1}}\|_{F} = \|{({\color{black} {Q}}\Delta AX^{-1})}^{T}{\color{black} {Q}} \Delta A\hat{{\color{black} {X}}^{-1}}\|_{F}$.}\\
Observe that
\begin{align}\label{inequlity}
  \sqrt{2}\|{\color{black} {Q}}^T\Delta A{\color{black} {X}}^{-1}\|_{F} + \frac{1}{\sqrt{2}}\left({\|{{\color{black} {Q}}}\Delta A {\color{black} {X}}^{-1}\|_{F}}^{2} + {\| \Delta {\color{black} {X}} {\color{black} {X}}^{-1}\|_{F}}^{2} \right)- \|\Delta {\color{black} {X}}{\color{black} {X}}^{-1}\|_{F} \geq 0.
 \end{align}
 The inequality \eqref{inequlity} can be considered a quadratic inequality on $\|\Delta {\color{black} {X}}{\color{black} {X}}^{-1}\|_{F}.$ Using \eqref{2.4}, we have
 \begin{align*}
\Sigma \equiv &(-1)^2 - 4 \times \frac{1}{\sqrt{2}} \times \left( \sqrt{2}\left\|{\color{black} {Q}}^T  \Delta A {\color{black} {X}}^{-1}\right\|_F + \frac{1}{\sqrt{2}}\left\|{\color{black} {Q}}  A \Delta {\color{black} {X}}^{-1}\right\|_F^2 \right)\\
 = &  1 - 4 \left\| {\color{black} {Q}}^T  \Delta A {\color{black} {X}}^{-1} \right\|_F - 2 \left\| {\color{black} {Q}}  A \Delta {\color{black} {X}}^{-1} \right\|_F^2 \geq 0.
\end{align*}
So,
\begin{align*}
  \| \Delta {\color{black} {X}}^{-1}{\color{black} {X}} \|_F \leq \frac{1}{\sqrt{2}} (1 - \sqrt{\Sigma}) \quad \text{or} \quad \|\Delta {\color{black} {X}}^{-1}  {\color{black} {X}}\|_F \geq \frac{1}{\sqrt{2}} (1 + \sqrt{\Sigma}).
\end{align*}
Therefore, $\frac{1}{\sqrt{2}} (1 - \sqrt{\Sigma}), \frac{1}{\sqrt{2}} (1 + \sqrt{\Sigma}), \text{ and } \| \Delta {\color{black} {X}}^{-1}{\color{black} {X}} \|_F $  are all continuous functions of the entries of  $\Delta A, \text{ and } \| \Delta {\color{black} {X}} \| \rightarrow 0 \text{ when } \Delta A \rightarrow 0 $,  we have
\begin{align}\label{2.8}
 & \|\Delta {\color{black} {X}}{\color{black} {X}}^{-1}\|_{F} \leq \frac{1}{\sqrt{2}} ( 1 - \sqrt{1 - 4\|{{\color{black} {Q}}^{T}} \Delta A{\color{black} {X}}^{-1}\|_{F}  - 2{\|{\color{black} {Q}}\Delta A{\color{black} {X}}^{-1}\|^{2}}_{F}}) < \frac{1}{\sqrt{2}}.
 \end{align}
Noting that ${\color{black} {X}} = D_{n}{\hat{{\color{black} {X}}}}$, where $D_{n} \in \mathbb{D}_{n},$
\begin{align*}
 & \Delta {\color{black} {X}}\hat{{\color{black} {X}}^{-1}} = {\rm {{\color{black}{upx}}}} [( {\color{black} {Q}}^T\Delta A\hat{{\color{black} {X}}^{-1}}) +  {D_{n}}^{-1}({\color{black} {Q}}^T\Delta A\hat{{\color{black} {X}}^{-1}})^{T}D_{n}]\nonumber\\
 & \quad\quad\quad\qquad + {\rm {{\color{black}{upx}}}}[ {\color{black} {X}}^{-T}\Delta A^{T}\Delta A\hat{{\color{black} {X}}^{-1}}] - {\rm {{\color{black}{upx}}}}[{\color{black} {X}}^{-T}\Delta {\color{black} {X}}^{T}\Delta {\color{black} {X}}\hat{{\color{black} {X}}^{-1}}].
 \end{align*}
Using the Frobenius norm on the preceding equation and {\color{black} {using \eqref{1.11}, \eqref{1.5}, and \eqref{1.3}, it yields to}}
\begin{align}
& \|\Delta {\color{black} {X}}\hat{{\color{black} {X}}^{-1}}\|_{F} \leq \sqrt{1 + {\varsigma^{2}}_{D_{n}}} \| {\color{black} {Q}}^T\Delta A\hat{{\color{black} {X}}^{-1}}\|_{F} + \|{\color{black} {X}}^{-T}\Delta A^{T}\Delta A\hat{{\color{black} {X}}^{-1}}\|_{F}+ \|{\color{black} {X}}^{-T}\Delta {\color{black} {X}}^{T}\Delta {\color{black} {X}}\hat{{\color{black} {X}}^{-1}}\|_{F},\nonumber \\
& \leq  \sqrt{1 + {\varsigma^{2}}_{D_{n}}} \| {\color{black} {Q}}^T\Delta A\hat{{\color{black} {X}}^{-1}}\|_{F} + \|{\color{black} {Q}} \Delta A {\color{black} {X}}^{-1}\|_{F}\|{\color{black} {Q}} \Delta A \hat{{\color{black} {X}}^{-1}}\|_{F}+ \|\Delta {\color{black} {X}}{\color{black} {X}}^{-1}\|_{F}\|\Delta {\color{black} {X}}\hat{{\color{black} {X}}^{-1}}\|_{F},
\end{align}\label{2.9}
which combined with the second inequality of \eqref{2.8}, \eqref{1.3}, \eqref{2.4}, and the fact that $\sqrt{1 + {\varsigma^{2}}_{D_{n}}} \geq 1,$ results in
\begin{align*}
& \|\Delta {\color{black} {X}}\hat{{\color{black} {X}}^{-1}}\|_{F} \leq \sqrt{1 + {\varsigma^{2}}_{D_{n}}} \| {\color{black} {Q}}^T\Delta A\hat{{\color{black} {X}}^{-1}}\|_{F} + (\sqrt{3/2} - 1)\|{\color{black} {Q}} \Delta A \hat{{\color{black} {X}}^{-1}}\|_{F} + \frac{1}{\sqrt{2}} \|\Delta {\color{black} {X}}\hat{{\color{black} {X}}^{-1}}\|_{F},  \\
&\qquad\quad\quad\quad\quad \leq \frac{ \sqrt{3}\sqrt{1 + {\varsigma^{2}}_{D_{n}}} \|{{\color{black} {Q}}}\|_{2}\|\Delta A\|_{F}\|\hat{{\color{black} {X}}}^{-1}\|_{2}}{\sqrt{2} - 1},  \\
&\qquad\quad\quad\quad\quad \leq (\sqrt{6} + \sqrt{3})\sqrt{1 + {\varsigma^{2}}_{D_{n}}} \|\Delta A\|_{F}\|\hat{{\color{black} {X}}}^{-1}\|_{2}\|{\color{black} {Q}}\|_{2}.
\end{align*}
We obtain the bound \eqref{2.2}, using the fact  $\|\Delta {\color{black} {X}}\|_{F} = \|\Delta {\color{black} {X}} \hat{{\color{black} {X}}}^{-1}\hat{{\color{black} {X}}}\|_{F} \leq  \|\Delta {\color{black} {X}}\hat{{\color{black} {X}}^{-1}}\|_{F} \|\hat{{\color{black} {X}}}\|_{2}.$ \\

Next, we are going to {\color{black}{derive}} the bound \eqref{2.3}. Specifically, we have
\begin{align*}
\Delta A = \Delta {\color{black}{QX}} + {({\color{black} {Q}} + \Delta {\color{black} {Q}})}\Delta {\color{black} {X}},
\end{align*}
which demonstrates
\begin{align}\label{2.10}
 & \Delta {\color{black} {Q}} = \Delta A {\color{black} {X}}^{-1} - ({\color{black} {Q}} + \Delta {\color{black} {Q}}) \Delta {\color{black}{X}} {\color{black} {X}}^{-1}.
  \end{align}
Pre-multiplying by ${\color{black} {Q}}^{T}$ and using the fact ${\color{black} {Q}}^{T}{\color{black} {Q}}= I_{n}$, we have
\begin{align*}
& {\color{black} {Q}}^{T}\Delta {\color{black} {Q}} = {\color{black} {Q}}^{T} \Delta A {\color{black} {X}}^{-1} - \Delta {\color{black} {X}} {\color{black} {X}}^{-1} - {\color{black} {Q}}^{T}\Delta {\color{black} {Q}} \Delta {\color{black} {X}} {\color{black} {X}}^{-1},
\end{align*}
which, together with \eqref{2.6}, yields
\begin{align}\label{2.11}
& {\color{black} {Q}}^{T}\Delta {\color{black} {Q}} = {\color{black} {Q}}^{T} \Delta A {\color{black} {X}}^{-1} - {\rm {\color{black}{upx}}} [ {\color{black} {Q}}^T\Delta A{\color{black} {X}}^{-1} +  ({{\color{black} {Q}}^T\Delta A{\color{black} {X}}^{-1} })^{T}]\nonumber\\
&\quad\quad\quad- {\rm {\color{black}{upx}}} [{\color{black} {X}}^{-T}\Delta A^{T}\Delta A{\color{black} {X}}^{-1}] + {\rm {\color{black}{upx}}}[ {\color{black} {X}}^{-T}\Delta {\color{black} {X}}^{T}\Delta {\color{black} {X}}{\color{black} {X}}^{-1}] - {\color{black} {Q}}^{T}\Delta {\color{black} {Q}} \Delta {\color{black} {X}} {\color{black} {X}}^{-1}.
\end{align}
Noting that $A - {\rm {\color{black}{up}}}(A) = {\rm {\color{black}{low}}}(A)$  and  ${\rm {\color{black}{low}}}(A)^{T} = {\rm {\color{black}{up}}}(A^{T})$, we have
\begin{align}\label{2.12}
&  {\color{black} {Q}}^{T}\Delta {\color{black} {Q}} = {\rm {\color{black}{lowx}}}({\color{black} {Q}}^T\Delta A{\color{black} {X}}^{-1}) - [{\rm {\color{black}{lowx}}}({\color{black} {Q}}^T\Delta A{\color{black} {X}}^{-1})]^{T}\nonumber\\
&\quad\quad\quad- {\rm {\color{black}{upx}}} [{\color{black} {X}}^{-T}\Delta A^{T}{\color{black} {Q}}^{T}{\color{black} {Q}}\Delta A{\color{black} {X}}^{-1}] + {\rm {\color{black}{upx}}}[ {\color{black} {X}}^{-T}\Delta {\color{black} {X}}^{T}\Delta {\color{black} {X}}{\color{black} {X}}^{-1}] - {\color{black} {Q}}^{T}\Delta {\color{black} {Q}} \Delta {\color{black} {X}} {\color{black} {X}}^{-1}.
\end{align}
Let $\hat{{\color{black} {Q}}} = {\color{black} {Q}}{D_{n}}^{-1}$ and $ {\color{black} {X}} = D_{n}\hat{{\color{black} {X}}}. $ By taking the Frobenius norm of \eqref{2.12} with \eqref{1.11}, \eqref{1.6}, and \eqref{1.3}, we get
\begin{align}\label{2.13}
&\| {\color{black} {Q}}^{T}\Delta {\color{black} {Q}}\|_{F} \leq \sqrt{2}\varsigma D_{n}\|\hat{{\color{black} {Q}}}^{T} \Delta A \hat{{\color{black} {X}}}^{-1}\|_{F} + \frac{1}{\sqrt{2}} {\|{\color{black} {Q}}^{T}\Delta A {\color{black} {X}}^{-1}\|_{F}}^{2}
+ \frac{1}{\sqrt{2}}{\|\Delta {\color{black} {X}}{\color{black} {X}}^{-1}\|_{F}}^{2} + \|{\color{black} {Q}}^{T}\Delta {\color{black} {Q}}\|_{F} \|\Delta {\color{black} {X}}{\color{black} {X}}^{-1}\|_{F}.
  \end{align}
From the first inequality of \eqref{2.8}, we obtain
\begin{align*}
 \|\Delta {\color{black} {X}}{\color{black} {X}}^{-1}\|_{F} & \leq \frac{1}{\sqrt{2}} \frac{4\|{\color{black} {Q}}^{T}\Delta A {\color{black} {X}}^{-1}\|_{F} + 2{\|Q
^{T}\Delta A {\color{black} {X}}^{-1}\|_{F}}^{2}}{1 + \sqrt{1 - 4\|{\color{black} {Q}}^{T}\Delta A {\color{black} {X}}^{-1}\|_{F} - 2{\|{\color{black} {Q}}^{T}\Delta A {\color{black} {X}}^{-1}\|_{F}}^{2} }} \\
& \leq \frac{1}{\sqrt{2}} (4\|{\color{black} {Q}}^{T}\Delta A {\color{black} {X}}^{-1}\|_{F} + 2{\|{\color{black} {Q}}^{T}\Delta A {\color{black} {X}}^{-1}\|_{F}}^{2}).
   \end{align*}
Squaring the above inequality and using \eqref{2.4} leads to
\begin{align}\label{2.14}
 {\|\Delta {\color{black} {X}}{\color{black} {X}}^{-1}\|_{F}}^{2} \leq (5 + 2\sqrt{6}) {\|{\color{black} {Q}}^{T}\Delta A {\color{black} {X}}^{-1}\|^{2}}_{F}.
\end{align}
Using the inequality \eqref{2.13} together with \eqref{2.4} and {\color{black}{the}} second inequality of \eqref{2.8} results in
\begin{align}\label{2.15}
 \|{\color{black} {Q}}^{T}\Delta {\color{black} {Q}}\|_{F}\leq (2 \sqrt{2} + 2) \varsigma_ {D_{n}} \|\hat{{\color{black} {Q}}}^{T} \Delta A \hat{{\color{black} {X}}}^{-1}\|_{F} + (2\sqrt{3} + \sqrt{6})\|{\color{black} {Q}}^{T}\Delta A {\color{black} {X}}^{-1}\|_{F},
\end{align}
which, combined with the fact that
\begin{align*}
& \|\Delta {\color{black} {Q}}\|_{F} = \|{\color{black} {Q}}^{T}{\color{black} {Q}}\Delta {\color{black} {Q}}\|_{F} \leq \|{\color{black} {Q}}^{T}\Delta {\color{black} {Q}}\|_{F}\|{\color{black} {Q}}\|_{2},
\end{align*}
yields the desired bound \eqref{2.3}.
}
\end{proof}
\begin{remark}
{\rm {\color{black} From the first inequality of \eqref{2.8}, together with \eqref{2.4} and \eqref{2.10}, we} obtain rigorous perturbation bounds listed below:}
\begin{align}
\frac{\|\Delta {\color{black} {X}}\|_{F}}{\|{\color{black} {X}}\|_{2}}& \leq \frac{\sqrt{2}(\mathop {\inf}\limits_{D_{n} \in \mathbb{D}_n} \sqrt{1 + {\varsigma^{2}}_{D_{n}}}\kappa_{2} (D^{-1}{\color{black} {X}}))\left(\frac{\|{\color{black} {Q}}^{T}\Delta A\|_{F}}{\|A\|_{F}} + \kappa_{2}(A)\frac{{\|\Delta A\|^{2}}_{F}}{{\|A\|^{2}}_{2}}\right)} {\sqrt{2} - 1 + \sqrt{1 -4\kappa_{2}(A)\frac{\|\Delta A\|_{F}}{\|A\|_{2}} - 2{\kappa^{2}}_{2}(A) \frac{{\|\Delta A\|^{2}}_{F}}{{\|A\|^{2}}_{2}}}}, \label{2.16}\\
&\leq \frac{\sqrt{3}(\mathop {\inf}\limits_{D_{n} \in \mathbb{D}_n} \sqrt{1 + {\varsigma^{2}}_{D_{n}}}\kappa_{2} (D^{-1}{\color{black} {X}}))\frac{\|\Delta A\|_{F}}{\|A\|_{2}} }{\sqrt{2} - 1 + \sqrt{1 - 4\kappa_{2}(A) \frac{\|\Delta A \|_{F}}{\|A\|_{2}} - 2{\kappa^{2}}_{2}(A) \frac{{\|\Delta A\|^{2}}_{F}}{{\|A\|}^{2}}_{2} }}.\label{2.17}
\end{align}
{\rm The bound \eqref{2.17} is stronger than \eqref{2.16}, which, in turn, is stronger than \eqref{2.2}. However, the
above bounds have a more complicated form.}
\end{remark}

\begin{remark}
{\rm The first-order perturbation bound for the factor ${\color{black} {X}}$ with norm-wise perturbation can be obtained by neglecting the higher-order perturbation terms from \eqref{2.17}}
\begin{align}\label{2.18}
\|\Delta {\color{black} {X}}\|_{F} \lesssim   inf_{D\in{D}_{n}}\left[\sqrt{1 + {\varsigma^{2}}_{D_{n}}} \kappa_{2} (D^{-1}{\color{black} {X}})\right]\|{\color{black} {Q}}\|_{2}\|\Delta A\|_{F},
\end{align}
{\rm which is the same as the one given in {\color{black}{\cite{lv2022perturbation}}}. The difference between \eqref{2.2} and \eqref{2.18} is the constant $(\sqrt{6} + \sqrt{3})$.}
\end{remark}
\subsection{Rigorous Perturbation Bounds by Using Modified Matrix-Vector Equation Approach}
{\color{black} In this subsection, we combine the modiﬁed matrix-vector equation approach \cite{H1, H2, H3}, the technique of Lyapunov majorant function (e.g., \cite[Chapter 5]{konstantinov2003perturbation}), and the Banach ﬁxed point theorem (e.g., \cite[Appendix D]{konstantinov2003perturbation}), to derive the rigorous perturbation bounds for the ${\color {black}{QX}}$ factorization for centrosymmetric matrices.}
\begin{theorem}\label{thm2.2}
Under the assumptions in Theorem \ref{thm2.1}, if
\begin{align}\label{2.19}
  {\|H_{{\color{black} {X}}}\|_{2}}(\|G_{{\color{black} {X}}}\|_{2}\delta + \|H_{{\color{black} {X}}}\|_{2}\delta^{2} ) < \frac{1}{4},
\end{align}
then {\color{black} the} ${\color {black} {QX}}$ factorization \eqref{2.1} is unique for  $ A + \Delta A $  and \\
\begin{align}
\| \Delta {\color{black} {X}}\|_{F}  \leq & \frac{2(\|G_{{\color{black} {X}}}\|_{2}\delta + \|H_{{\color{black} {X}}}\|_{2}{\delta}^{2})}{1 + \sqrt{1 - 4\|H_{{\color{black} {X}}}\|_{2}(\|G_{{\color{black} {X}}}\|_{2}\delta + \|H_{{\color{black} {X}}}\|_{2}{\delta}^{2})}}, \label{2.20}\\
 \leq  & 2(\|G_{{\color{black} {X}}}\|_{2} \delta + \|H_{{\color{black} {X}}}\|_{2}{\delta}^{2}), \label{2.21}\\
 \leq  &  (1 + 2\|G_{{\color{black} {X}}}\|_{2}) \delta, \label{2.22}
\end{align}
where $\delta = \|\Delta A\|_{F},\! G_{{\color{black} {X}}} = M_{\rm {\color{black}{ \rm xvec}}}( {\color{black} {X}}^{T} \otimes I_{n}) M_{\rm {\color{black}{upx}}}[({\color{black} {X}}^{-T} \otimes Q
 ^{T}) + ({\color{black} {Q}}^{T} \otimes {\color{black} {X}}^{-T})\Pi_{m,n}],$ and $
 H_{{\color{black} {X}}} = M_{{\color{black}{ \rm xvec}}}( {\color{black} {X}}^{T} \otimes I_{n})M_{\rm{\color{black}{upx}}}({\color{black} {X}}^{-T} \otimes {\color{black} {X}}^{-T}).$

\end{theorem}
\begin{proof}
Assume that the matrices in \eqref{1.2} are perturbed, i.e.,
\begin{align*}
{A \rightarrow \Delta A,\,\,\,{\color{black} {Q}}\rightarrow \Delta {\color{black} {Q}},\,\,\ {\color{black} {X}}\rightarrow \Delta {\color{black} {X}},}
\end{align*}
and the perturbed version of \eqref{1.2} is given by
\begin{align}\label{2.23}
(A + \Delta A)^{T}(A+\Delta A) = ({\color{black} {X}} + \Delta {\color{black} {X}})^{T}({\color{black} {X}}+ \Delta {\color{black} {X}}).
\end{align}
To derive the strong rigorous pertubation bound for the factor ${\color{black} {X}}$, from \eqref{2.6}, we have
\begin{align}\label{2.24}
{\rm {\color{black}{vec}}}(\Delta  {\color{black} {X}}) & = ( {\color{black} {X}}^{T} \otimes I_{n}) M_{\rm {\color{black}{upx}}}[({\color{black} {X}}^{-T} \otimes {\color{black} {Q}}^{T}){\rm {\color{black}{vec}}}(\Delta A) + ({{\color{black} {Q}}^{T}} \otimes {\color{black} {X}}^{-T})\Pi_{m,n}{\rm {\color{black}{vec}}}(\Delta A)]\nonumber\\
& + ( {\color{black} {X}}^{T} \otimes I_{n})M_{\rm {\color{black}{upx}}}[({\color{black} {X}}^{-T} \otimes {\color{black} {X}}^{-T}) {\rm {\color{black}{vec}}} (\Delta A_{T}\Delta A - \Delta {\color{black} {X}}_{T}\Delta {\color{black} {X}}) ].
\end{align}
As $\Delta {\color{black} {X}} $ is an ${\color{black} {X}}$-type matrix, therefore \eqref{1.8} and \eqref{1.4} together give
\begin{align}\label{2.25}
\Delta {\color{black} {X}}{\color{black} {X}}^{-1} = {\rm {\color{black}{upx}}} [ {\color{black} {Q}}^T\Delta A{\color{black} {X}}^{-1} +  ({{\color{black} {Q}}^T\Delta A{\color{black} {X}}^{-1} })^{T} + {\color{black} {X}}^{-T}\Delta A^{T}\Delta A{\color{black} {X}}^{-1} -  {\color{black} {X}}^{-T}\Delta {\color{black} {X}}^{T}\Delta {\color{black} {X}}{\color{black} {X}}^{-1}]
\end{align}
and
\begin{align}\label{2.26}
{\rm {\color{black}{vec}}} (\Delta {\color{black} {X}}) & = {\rm {\color{black}{vec}}}\left(\rm {\color{black}{utx}}(\Delta {\color{black} {X}})\right) = M_{\rm {\color{black}{utx}}}{\rm {\color{black}{vec}}}(\Delta {\color{black} {X}}) \nonumber\\
& = {M_{\rm {\color{black}{ \rm xvec}}}^{T}} M_{\rm {\color{black}{ \rm xvec}}}{\rm {\color{black}{vec}}}(\Delta {\color{black} {X}}) = {M_{\rm {\color{black}{ \rm xvec}}}^{T}}{\rm xvec}(\Delta {\color{black} {X}}).
\end{align}
Substituting \eqref{2.26} into \eqref{2.25}, and then multiplying by $ M_{\rm {\color{black}{ \rm xvec}}}$ from the left and using \eqref{1.8}, leads to,
\begin{align}\label{2.27}
& {\rm {\color{black}{ \rm xvec}}}(\Delta {\color{black} {X}}) = M_{\rm {\color{black}{ \rm xvec}}}( {\color{black} {X}}^{T} \otimes I_{n}) M_{\rm {\color{black}{upx}}}[({\color{black} {X}}^{-T} \otimes {\color{black} {Q}}^{T}){\rm {\color{black}{vec}}}(\Delta A) + ({\color{black} {Q}}^{T} \otimes {\color{black} {X}}^{-T})\Pi_{m,n}{\rm {\color{black}{vec}}} (\Delta A)]\nonumber\\
&\quad\quad\quad\quad\quad+ M_{\rm {\color{black}{ \rm xvec}}}( {\color{black} {X}}^{T} \otimes I_{n})M_{\rm {\color{black}{upx}}}[({\color{black} {X}}^{-T}\otimes {\color{black} {X}}^{-T}){\rm {\color{black}{vec}}}(\Delta A^{T}\Delta A - \Delta {\color{black} {X}}^{T}\Delta {\color{black} {X}})].
\end{align}
{\rm Conversely, multiplying on the left \eqref{2.27} by $M^{T}_{\rm {\color{black}{ \rm xvec}}}$ together with \eqref{2.26} and \eqref{1.8},}
\begin{align}\label{2.28}
&{\rm {\color{black}{ \rm xvec}}} (\Delta {\color{black} {X}}) = M_{\rm {\color{black}{utx}}} ( {\color{black} {X}}^{T} \otimes I_{n})M_{\rm {\color{black}{upx}}}[({\color{black} {X}}^{-T} \otimes {\color{black} {Q}}^{T}){\rm {\color{black}{ \rm xvec}}}(\Delta A) + ({\color{black} {Q}}^{T} \otimes {\color{black} {X}}^{-T})\Pi_{m,n}{\rm {\color{black}{ \rm xvec}}}(\Delta A)]\nonumber \\
&\quad\quad\quad\quad\quad+ M_{\rm {\color{black}{utx}}}( {\color{black} {X}}^{T} \otimes I_{n})M_{\rm {\color{black}{upx}}}[({\color{black} {X}}^{-T} \otimes {\color{black} {X}}^{-T}){\rm {\color{black}{ \rm xvec}}} (\Delta A^{T}\Delta A - \Delta {\color{black} {X}}^{T}\Delta {\color{black} {X}})].
\end{align}
{\rm We can verify from $M_{\rm {\color{black}{utx}}}$ and $M_{\rm {\color{black}{upx}}},$ i.e.,} $M_{\rm {\color{black}{utx}}}({\color{black} {X}}^{T} \otimes I_{n}) M_{\rm {\color{black}{upx}}} = ({\color{black} {X}}^{T} \otimes I_{n}) M_{\rm {\color{black}{upx}}}.$
{\rm Thus, \eqref{2.24} is equivalent to \eqref{2.27}.}
{\rm Applying '${\rm {\color{black}{ \rm xvec}}^{\dag}}$' to \eqref{2.28} results in}
\begin{align*}
&\Delta {\color{black} {X}} = {\rm {\color{black}{ \rm xvec}}}^{\dag}[G_{{\color{black} {X}}} {\rm {\color{black}{vec}}}(\Delta A) + H_{{\color{black} {X}}}{\rm {\color{black}{vec}}}(\Delta A^{T} \Delta A - \Delta {\color{black} {X}}^{T} \Delta {\color{black} {X}})].
\end{align*}
{\rm As a result, we can write the preceding equation as the operator equation for $\Delta {\color{black} {X}}$,} i.e.,
\begin{align}\label{2.29}
 \Delta {\color{black} {X}} & = \Psi(\Delta {\color{black} {X}}, \Delta A) \nonumber\\
& = {\rm {\color{black}{ \rm xvec}}}^{\dag}[G_{{\color{black} {X}}} {\rm {\color{black}{vec}}}(\Delta A) + H_{{\color{black} {X}}}{\rm {\color{black}{vec}}}(\Delta A^{T} \Delta A - \Delta {\color{black} {X}}^{T} \Delta {\color{black} {X}})].
\end{align}
{\rm Now we use the Lyapunov majorant function technique and the Banach fixed point principle to obtain the rigorous perturbation bound for $\Delta {\color{black} {X}},$ depending on {\color{black} the} operator equation \eqref{2.29}. Assume that $Z \in {U}_{n}$ {\color{black}has} the same structure as that of $\Delta {\color{black} {X}},$ and replacing it in the preceding equation with $Z$, we have }
\begin{align}\label{2.30}
Z  = \Psi(Z, \Delta A),
\end{align}
where $\Psi(Z, \Delta A)  = {\rm {\color{black}{ \rm xvec}}}^{\dag}[G_{{\color{black} {X}}} {\rm {\color{black}{vec}}}(\Delta A) + H_{{\color{black} {X}}}{\rm {\color{black}{vec}}}(\Delta A^{T} \Delta A -  Z^{T}Z)].$ {\rm Let} $\|Z\|_{F} \leq \rho $ \text{for some} $ \rho \geq 0 $. {\rm Then, noting \eqref{1.3}},
\begin{align*}
\|\Psi(Z, \Delta A)\|_{F} \leq \|G_{{\color{black} {X}}}\|_{2} \delta + \|H_{{\color{black} {X}}}\|_{2} {\delta}^{2} + \|H_{{\color{black} {X}}}\|_{2} \rho^{2}.
\end{align*}
{\rm Thus, we have the Lyapunov majorant function of the operator equation \eqref{2.30},}
\begin{align}\label{2.31}
 h(\rho, \delta) = \rho,   \quad  {\rm i.e.}, \quad    \|G_{{\color{black} {X}}}\|_{2} \delta + \|H_{{\color{black} {X}}}\|_{2} {\delta}^{2} + \|H_{{\color{black} {X}}}\|_{2} \rho^{2} = \rho.
\end{align}
{\rm Assume that,} $\delta \in \Omega = \{\delta \geq 0: 1 - 4\|H_{{\color{black} {X}}}\|_{2}(\|G_{{\color{black} {X}}}\|_{2}\delta + \|H_{{\color{black} {X}}}\|_{2}{\delta^{2}}) \geq 0\},$ {\rm then there are two roots to the Lyapunov majorant equation \eqref{2.31}, $\rho_{1}(\delta) \leq \rho_{2}(\delta) $ with}
\begin{align*}
\rho_{1}(\delta):= f_{1}(\delta) := \frac{2(\|G_{{\color{black} {X}}}\|_{2}\delta + \|H_{{\color{black} {X}}}\|_{2}{\delta}^{2})}{1 + \sqrt{1 - 4\|H_{{\color{black} {X}}}\|_{2}(\|G_{{\color{black} {X}}}\|_{2}\delta + \|H_{{\color{black} {X}}}\|_{2}{\delta}^{2})}}.
\end{align*}
{\rm Let the set $\mathfrak{B}(\delta)$ be $ \mathfrak{B}(\delta)= \{ Z \in {U}_{n} : \|Z\|_{F} \leq f_{1}(\delta) \} \subset R^{n \times n}$ is closed and convex. The operator $\Psi(., \Delta A)$ maps the set $\mathfrak{B}(\delta)$ into itself.} {\rm Furthermore, when } $\delta \in \Omega_{1} = \{\delta \geq 0: 1 - 4\|H_{{\color{black} {X}}}\|_{2}(\|G_{{\color{black} {X}}}\|_{2}\delta + \|H_{{\color{black} {X}}}\|_{2}{\delta^{2}}) > 0\},$ {\rm the derivative of the function} $h(\rho, \delta)$  {\rm relative to} $\rho$ at $f_{1}(\delta)$ {\rm satisfies}
\begin{align*}
 h'_{\rho}(f_{1}(\delta), \delta) = 1 - \sqrt{1 - 4\|H_{{\color{black} {X}}}\|_{2}(\|G_{{\color{black} {X}}}\|_{2}\delta + \|H_{{\color{black} {X}}}\|_{2}{\delta})\rho} \leq 1.
\end{align*}
{\rm Meanwhile, for} $ Z, \tilde{Z} \in \mathfrak{B}(\delta) $,
\begin{align*}
\|\Psi(Z, \Delta A) - \Psi(\tilde{Z}, \Delta A)\|_{F} \leq h'_{\rho}(f_{1}(\delta), \delta)\|Z - \tilde{Z}\|_{F}.
\end{align*}
{\rm The above facts mean that the operator $\Psi(., \Delta A)$ is contractive on the set $\mathfrak{B}(\delta)$ when $\delta \in {\Omega_{1}}$. According to the Banach fixed point theorem, the operator equation \eqref{2.30} has a unique solution in the set $\mathfrak{B}(\delta)$ for $\delta \in {\Omega_{1}}$, and so does the operator equation \eqref{2.29}, and then the matrix equation \eqref{2.1}. Then $ \|\Delta X \|_{F} \leq f_{1}(\delta)$ for $\delta \in {\Omega}_{1}$. It is sufficient to obtain bound \eqref{2.22}, which is obtained by using \eqref{2.19} and the fact}
\begin{align}\label{2.32}
2\|H_{{\color{black} {X}}}\|_{2}\|\Delta A\|_{F} \leq \sqrt{1 + {\|G_{{\color{black} {X}}}\|_{2}}^{2}} - \|G_{{\color{black} {X}}}\|_{2} < 1.
\end{align}
{\rm Using \eqref{2.32} into \eqref{2.21}, and noting the fact}
\begin{align*}
   & \sqrt{1 + {\|G_{{\color{black} {X}}}\|_{2}}^{2}} \leq 1 + \|G_{{\color{black} {X}}}\|_{2},
\end{align*}
{\rm we obtain the desired bound \eqref{2.22}.}
\end{proof}
\begin{remark}{\rm From \eqref{2.1}, we have,}
\begin{equation}\label{2.33}
\Delta A = {\color{black} {Q}} \Delta {\color{black} {X}} + \Delta {\color{black} {Q}} {\color{black} {X}} + \Delta {\color{black} {Q}} \Delta {\color{black} {X}}
\end{equation}
{\rm and by post-multiplying by ${\color{black} {X}}^{-1},$ we have}
\begin{align}\label{2.34}
 \Delta {\color{black} {Q}} & = \Delta A {\color{black} {X}}^{-1} - {\color{black} {Q}} \Delta {\color{black} {X}} {\color{black} {X}}^{-1} - \Delta {\color{black} {Q}} \Delta {\color{black} {X}} {\color{black} {X}}^{-1}.
\end{align}
{\rm Using the operator ${\rm {\color{black}{vec}}}$ on the preceding equation and utilizing \eqref{1.8}, \eqref{1.17}, and \eqref{1.13}, we obtain the following}
 \begin{align}\label{2.35}
{\rm {\color{black}{vec}}} (\Delta {\color{black} {Q}}) & = ({\color{black} {X}}^{-T} \otimes I_{n}){\rm {\color{black}{vec}}}(\Delta A) - ({\color{black} {X}}^{-T} \otimes L){\rm {\color{black}{vec}}}(\Delta {\color{black} {X}}) - ({\color{black} {X}}^{-T} \otimes I_{n}){\rm {\color{black}{vec}}}(\Delta {\color{black} {Q}} \Delta {\color{black} {X}}).
\end{align}
{\rm  With the help of \eqref{2.35}, and \eqref{2.24}, we get}
\begin{align*}
 \|\Delta {\color{black} {Q}} \|_{F} & \leq  \|G_{{\color{black} {Q}}}\|_{2}\|\Delta A\|_{F} + \|(I_{n}\otimes {\color{black} {Q}})M_{\rm {\color{black}{upx}}} ({\color{black} {X}}^{-T} \otimes {\color{black} {X}}^{-T})\|_{2} \left( \|{\Delta A\|^{2}}_{F} +  \|{\Delta {\color{black} {X}}\|^{2}}_{F} \right) \nonumber\\
& + \|{\color{black} {X}}^{-T} \otimes I_{n}\|_{2}\|\Delta {\color{black} {Q}}\|_{F}\|\Delta {\color{black} {X}}\|_{F},
  \end{align*}
{\rm where}
\begin{align*}
G_{{\color{black} {Q}}} = \left( {\color{black} {X}}^{T} \otimes I_{m}\right) - (I_{n} \otimes {\color{black} {Q}}) M_{\rm {\color{black}{upx}}}\left[({\color{black} {X}}^{-T} \otimes {\color{black} {Q}}^{T}) + ({\color{black} {Q}}^{T} \otimes {\color{black} {X}}^{-T})\Pi_{m,n}\right].
\end{align*}
{\rm By applying the previous result to the ${\color{black}{X}}$ factor, we also derive rigorous perturbation bounds for the ${\color{black}{Q}}$ factor.}
\begin{align*}
\|\Delta {\color{black} {Q}}\|_{F} & \leq \left( 2 + \sqrt{2} \right)\|(I_{n}\otimes {\color{black} {Q}})M_{\rm {\color{black}{upx}}} ({\color{black} {X}}^{-T} \otimes {\color{black} {X}}^{-T})\|_{2} \left( \|{\Delta A\|^{2}}_{F} +  \|{\Delta {\color{black} {X}}\|^{2}}_{F} \right) + \left( 2 + \sqrt{2} \right)\|G_{{\color{black} {Q}}}\|_{2}\|\Delta A\|_{F}, \nonumber\\
&  \leq \left( 2 + \sqrt{2} \right)\left[ \|G_{{\color{black} {Q}}}\|_{2} + \|{\color{black} {X}}^{-T} \otimes {\color{black} {Q}}\|_{2} (1 + \|G_{{\color{black} {X}}}\|_{2}) \right]\|\Delta A\|_{F}.
\end{align*}
{\rm These correspond to the first-order perturbation bounds. However, these bounds are relatively larger than the one given in \cite{lv2022perturbation}. As a result, we skip their detailed derivation.}
\end{remark}
\begin{remark}
{\rm According to {\eqref{2.32}}, the condition \eqref{2.19} can be simplified and strengthened to}
\begin{align}\label{2.36}
   & \|H_{{\color{black} {X}}}\|_{2}(1 + 2\|G_{{\color{black} {X}}}\|_{2})\|\Delta A\|_{F} < \frac{1}{2}.
\end{align}
\end{remark}
\begin{remark}
{\rm Now, we show that the bound \eqref{2.22} is sharper than \eqref{2.2}. For any $D \in \mathbb{D}_{n}$ and $A \in \mathbb{R}^{m \times n}$, using \eqref{1.3}, \eqref{1.14},\eqref{1.15}, and \eqref{1.16}, we have}\\
\begin{align}
\|G_{{\color{black} {X}}}\|_{2}  &= \|M_{\rm{{\color{black}{ \rm xvec}}}}({\color{black} {X}}^{T} \otimes I_{n})(D^{-1} \otimes I_{n})(D \otimes I_{n})M_{\rm{{\color{black}{upx}}}}[({\color{black} {X}}^{-T} \otimes {\color{black} {Q}}^{T}) + ({\color{black} {Q}}^{T} \otimes {\color{black} {X}}^{-T})\Pi_{m,n}]\|_{2} \nonumber \\
  &= \|M_{\rm{{\color{black}{ \rm xvec}}}}({\color{black} {X}}^{T}D^{-1} \otimes I_{n})M_{\rm{upx}}[(D{\color{black} {X}}^{-T} \otimes {\color{black} {Q}}^{T}) + (D{\color{black} {Q}}^{T} \otimes X
  {\color{black} {X}}^{-T})\Pi_{m,n}]\|_{2} \nonumber \\
  &\le \|{\color{black} {X}}^{T}D^{-1}\|_{2}\|M_{\rm{{\color{black}{upx}}}}[(D{\color{black} {X}}^{-T} \otimes {\color{black} {Q}}^{T}) + (D{\color{black} {Q}}^{T} \otimes {\color{black} {X}}^{-T})\Pi_{m,n}]\|_{2} \nonumber\\
  &= \|{\color{black} {X}}^{T}D^{-1}\|_{2}\mathop {\max }\limits_{\|\rm {{\color{black}{vec}}}(A)\|_{2} = 1}\|M_{\rm{{\color{black}{upx}}}}[(D{\color{black} {X}}^{-T} \otimes {\color{black} {Q}}^{T}) + (D{\color{black} {Q}}^{T} \otimes {\color{black} {X}}^{-T})\Pi_{m,n}]\rm{{\color{black}{vec}}}(A)\|_{2}. \label{e1}
\end{align}
{\rm Combining \eqref{1.17}, \eqref{1.9}, \eqref{1.13}, \eqref{1.11}, and \eqref{1.3}, we have the following}
\begin{align}
\|G_{{\color{black} {X}}}\|_{2}  &   \leq \|D^{-1}{\color{black} {X}}\|_{2}\mathop {\max }\limits_{\|\rm {{\color{black}{vec}}}(A)\|_{2} = 1}\|M_{\rm{{\color{black}{upx}}}}\rm {{\color{black}{vec}}}({\color{black} {Q}}^{T}A{\color{black} {X}}^{-1}D - ({\color{black} {X}}^{-T}A^{T}{\color{black} {Q}}D))\|_{2} \nonumber \\
   &= \|D^{-1}{\color{black} {X}}\|_{2}\mathop {\max }\limits_{\|\rm {{\color{black}{vec}}}(A)\|_{2} = 1}\|\rm {{\color{black}{vec}}}(\rm{{\color{black}{upx}}}({\color{black} {Q}}^{T}A{\color{black} {X}}^{-1}D - D^{-1}({\color{black} {Q}}^{T}A{\color{black} {X}}^{-1}D)^{T})D)\|_{2} \nonumber \\
   &= \|D^{-1}{\color{black} {X}}\|_{2}\mathop {\max }\limits_{\|A\|_{F} = 1}\|\rm{{\color{black}{upx}}}({\color{black} {Q}}^{T}A{\color{black} {X}}^{-1}D - D^{-1}({\color{black} {Q}}^{T}A{\color{black} {X}}^{-1}D)^{T})D\|_{2} \nonumber \\
   &\leq \|D^{-1}{\color{black} {X}}\|_{2}\mathop {\max }\limits_{\|A\|_{F} = 1}\sqrt{1 + \varsigma^{2}_{D}} \|{\color{black} {Q}}^{T}A{\color{black} {X}}^{-1}D\|_{F} \nonumber \\
   &\leq \|D^{-1}{\color{black} {X}}\|_{2}\mathop {\max }\limits_{\|A\|_{F} = 1}\sqrt{1 + \varsigma^{2}_{D}}\|{\color{black} {Q}}\|_{2}\|{\color{black} {X}}^{-1}D\|_{2}\|A\|_{F} \nonumber \\
   &\leq \sqrt{1 + \varsigma^{2}_{D}} \|D^{-1}{\color{black} {X}}\|_{2}\|{\color{black} {X}}^{-1}D\|_{2} \nonumber \\
   &\leq \sqrt{1 + \varsigma^{2}_{D}}\kappa_{2}(D^{-1}{\color{black} {X}}).\label{e2}
\end{align}
Thus, substituting \eqref{e2} into \eqref{e1} yields
\begin{align*}
\|G_{{\color{black} {X}}}\|_{2} & \leq \mathop {\inf}\limits_{D_{n} \in \mathbb{D}_n}\left(\sqrt{1 + \varsigma^{2}_{D}}\kappa_{2}(D^{-1}{\color{black} {X}})\right).
\end{align*}
{\rm Thus, we have proved that the bound \eqref{2.22} is indeed tighter than \eqref{2.2}. Similarly, the bound \eqref{2.20} is also tighter than \eqref{2.2}. }
\end{remark}

\subsection{Component-wise Perturbation}
Consider the component-wise perturbation in the matrix, as described in \cite{higham1994survey, samar2021refined}
\begin{align}\label{2.38}
   | \Delta A | \leq \epsilon K | A | ,  \quad   K = (k_{ij})\in {\rm \mathbb{R}}^{m \times m}, \quad   0 \leq k_{ij} \leq 1 ,
\end{align}
where, $\epsilon \geq 0$ is a small scalar and $ \Delta A $ is the perturbation matrix. Using a matrix equation technique, we now examine component-wise, rigorous perturbation bounds for the ${\color{black} {Q}}$ and ${\color{black} {X}}$ factors. These bounds are derived similarly to those in \cite{xie2015sensitivity}.

\begin{theorem}\label{thm 2.3}
If $A$ has a unique ${\color{black}{QX}}$ factorization, and suppose the conditions of Theorem \ref{thm2.1} hold and $\Delta A$ satisfies \eqref{2.38}.  Then the following inequality is satisfied
\begin{align}\label{2.39}
& {\parallel \mid {\color{black} {Q}}^{T} \mid \mid K \mid {\color{black} {X}}^{-1} \parallel }_{F} {\rm cond}({\color{black} {X}}) \epsilon < \frac{1}{\sqrt{6} + 2}.
\end{align}
Moreover, $A + \Delta A$ also possesses the ${\color{black}{QX}}$ factorization, and the following bounds hold:
\begin{align}\label{2.40}
  \| \Delta {\color{black} {X}} \|_{F} \leq \frac {\mathop {\inf}\limits_{D_{n} \in \mathbb{D}_n}\||{\color{black} {X}}| {\color{black} {X}}^{-1}D_{n}\|_{2}(\sqrt{2 + 2{\varsigma^{2}}_{D}}\||{\color{black} {Q}}^{T}| K|{\color{black} {Q}}|\|_{F} + (\sqrt{3}-\sqrt{2})\||{\color{black} {Q}}^{T}| K|{\color{black} {Q}}|\|_{F})\epsilon}{\sqrt{2} - 1},
\end{align}
\begin{align}\label{2.41}
   \| \Delta {\color{black} {X}} \|_{F} \leq (\sqrt{3} + \sqrt{6}) \left( \mathop {\inf}\limits_{D_{n} \in \mathbb{D}_n} \sqrt{1 + {\varsigma^{2}}_{D}} \||{\color{black} {X}}||{\color{black} {X}}^{-1}|D_{n}\|_{2} \|{{D}_{n}}^{-1}{\color{black} {X}}\|_{2}\right)\||{\color{black} {Q}}^{T}| K|{\color{black} {Q}}|\|_{F}\epsilon,
\end{align}
\begin{align}\label{2.42}
  \| \Delta {\color{black} {Q}} \|_{F} \leq (\sqrt{6} + 2 + 2\sqrt{2} + 2\sqrt{3}) \||{\color{black} {Q}}^{T}| K|{\color{black} {Q}}|\|_{F}{ \rm cond}({\color{black} {X}}) \epsilon,
\end{align}
where ${\rm cond}({\color{black} {X}})$ is defined as ${\||{\color{black} {X}}||{\color{black} {X}}^{-1}|\|}_{2}$.
\end{theorem}

\begin{proof}
{\rm From equation \eqref{2.6}, we have}
\begin{align}\label{2.43}
  \Delta {\color{black} {X}}{\color{black} {X}}^{-1} = {\rm{\color{black}{upx}}}[{\color{black} {Q}}^{T} \Delta A {\color{black} {X}}^{-1} + ({\color{black} {Q}}^{T} \Delta A {\color{black} {X}}^{-1})^{T} ] + {\rm{\color{black}{upx}}}[{\color{black} {X}}^{-T}({\Delta A}^{T} \Delta A - {\Delta {\color{black} {X}}}^{T} \Delta {\color{black} {X}}){\color{black} {X}}^{-1}].
\end{align}
{\rm Using the Frobenius norm on the above equation and by \eqref{1.11}, with $D_{n} = I_{n},$ gives us}
\begin{align*}
\|\Delta {\color{black} {X}}{\color{black} {X}}^{-1}\|_{F} & \leq \sqrt{2} \|{\color{black} {Q}}^{T} \Delta A {\color{black} {X}}^{-1}\|_{F} + \frac{1}{\sqrt{2}}\|X
^{-1} \Delta A^{T}{\color{black} {Q}}^{T}{\color{black} {Q}}  \Delta A {\color{black} {X}}^{-1} \|_{F} + \frac{1}{\sqrt{2}}{\|\Delta {\color{black} {X}}{\color{black} {X}}^{-1}\|}^{2}_{F} \\
& \|\Delta {\color{black} {X}}{\color{black} {X}}^{-1}\|_{F} \leq \sqrt{2} \|{\color{black} {Q}}^{T} \Delta A {\color{black} {X}}^{-1}\|_{F} + \frac{1}{\sqrt{2}}{\|{\color{black} {Q}}  \Delta A {\color{black} {X}}^{-1} \|}^{2}_{F} + \frac{1}{\sqrt{2}}{\|\Delta {\color{black} {X}}{\color{black} {X}}^{-1}\|}^{2}_{F}.
\end{align*}
{\rm Therefore,}
\begin{align}\label{2.44}
   &  \frac{1}{\sqrt{2}}{\|\Delta {\color{black} {X}}{\color{black} {X}}^{-1}\|}^{2}_{F} - \|\Delta {\color{black} {X}}{\color{black} {X}}^{-1}\|_{F} + \sqrt{2} \|{\color{black} {Q}}^{T} \Delta A {\color{black} {X}}^{-1}\|_{F} + \frac{1}{\sqrt{2}}{\|{\color{black} {Q}}  \Delta A {\color{black} {X}}^{-1} \|}^{2}_{F} \geq 0.
\end{align}
{\rm The inequality presented above can be regarded as a quadratic inequality on $\|\Delta {\color{black} {X}}{\color{black} {X}}^{-1}\|_{F}$. Using \eqref{2.39},  we have the following}
\begin{align*}
\mathfrak{\gamma} & \equiv {(-1)}^{2} - 4 \times \frac{1}{\sqrt{2}} \times \left ( \sqrt{2} \|{\color{black} {Q}}^{T} \Delta A {\color{black} {X}}^{-1}\|_{F} + \frac{1}{\sqrt{2}}{\|{\color{black} {Q}}  \Delta A {\color{black} {X}}^{-1} \|}^{2}_{F} \right ) \\
& = 1-4 \|{\color{black} {Q}}^{T} \Delta A {\color{black} {X}}^{-1}\|_{F} - 2{\|{\color{black} {Q}}  \Delta A {\color{black} {X}}^{-1} \|}^{2}_{F} \geq 0.
\end{align*}
{\rm As a result, we have $\|\Delta {\color{black} {X}}{\color{black} {X}}^{-1}\|_{F} \leq \frac{1}{\sqrt{2}} (1 - \sqrt{\mathfrak{\gamma}})$ or $\|\Delta {\color{black} {X}}{\color{black} {X}}^{-1}\|_{F} \geq \frac{1}{\sqrt{2}} (1 + \sqrt{\mathfrak{\gamma}}).$ Since $ \frac{1}{\sqrt{2}} (1 - \sqrt{\mathfrak{\gamma}})$ or $ \frac{1}{\sqrt{2}}(1 + \sqrt{\mathfrak{\gamma}}),$ and $\|\Delta {\color{black} {X}}{\color{black} {X}}^{-1}\|_{F}  $ are continuous functions of the entries of $\Delta A $, and $\Delta {\color{black} {X}} \rightarrow 0 $ as $\Delta A \rightarrow $ 0, we must have}
 \begin{align}\label{2.ab}
 \|\Delta {\color{black} {X}}{\color{black} {X}}^{-1}\|_{F} \leq \frac{1}{\sqrt{2}}\left (1 - \sqrt{1-4 \|{\color{black} {Q}}^{T} \Delta A {\color{black} {X}}^{-1}\|_{F} - 2{\|{\color{black} {Q}}  \Delta A {\color{black} {X}}^{-1} \|}^{2}_{F}}\right ) < \frac{1}{\sqrt{2}}.
 \end{align}
{\rm From right-multiplying \eqref{2.43} by $D_{n}$, where $D_{n} \in \mathbb{D}_{n}$, we obtain the following}
\begin{align*}
\Delta {\color{black} {X}}{\color{black} {X}}^{-1}D_{n} & = {\rm{\color{black}{upx}}}[({\color{black} {Q}}^{T} \Delta A {\color{black} {X}}^{-1}D_{n}) + {D^{-1}}_{n}(D_{n}{\color{black} {X}}^{-T}\Delta A^{T} {\color{black} {Q}})D_{n} ] \nonumber \\
& + {\rm {\color{black}{upx}}}[{\color{black} {X}}^{-T}({\Delta A}^{T} \Delta A - {\Delta {\color{black} {X}}}^{T} \Delta {\color{black} {X}}){\color{black} {X}}^{-1}D_{n}].
\end{align*}
{\rm By taking the Frobenius norm and utilizing \eqref{1.11}, we have}
\begin{align*}
\|\Delta {\color{black} {X}}{\color{black} {X}}^{-1}D_{n}\|_{F}  & \leq \sqrt{1 + {\varsigma}^{2}D_{n}}\|{\color{black} {Q}}^{T} \Delta A {\color{black} {X}}^{-1}D_{n}\|_{F} + \|{\color{black} {X}}^{-T}\Delta A^{T} {\color{black} {Q}}^{T}{\color{black} {Q}} \Delta A {\color{black} {X}}^{-1} D_{n}\|_{F} \nonumber \\
 &+ \|{\color{black} {X}}^{-T}\Delta {\color{black} {X}}^{T} {\color{black} {Q}}^{T}{\color{black} {Q}} \Delta {\color{black} {X}} {\color{black} {X}}^{-1} D_{n}\| \\
&\leq \sqrt{1 + {\varsigma}^{2}D_{n}}\|{\color{black} {Q}}^{T} \Delta A {\color{black} {X}}^{-1}D_{n}\|_{F} +  \|{\color{black} {Q}} \Delta A {\color{black} {X}}^{-1}\|_{F}\|{\color{black} {Q}} \Delta A {\color{black} {X}}^{-1}D_{n}\|_{F}\nonumber\\
& + \|\Delta {\color{black} {X}}{\color{black} {X}}^{-1}\|_{F}\|\Delta {\color{black} {X}}{\color{black} {X}}^{-1}D_{n}\|_{F}.
\end{align*}
{\rm Now using the fact $\|\Delta {\color{black} {X}}{\color{black} {X}}^{-1}\|_{F} \leq \frac{1}{\sqrt{2}}$, we get}
\begin{align*}
\|\Delta {\color{black} {X}}{\color{black} {X}}^{-1}D_{n}\|_{F} \leq \frac{\sqrt{2 + 2\varsigma^{2}D_{n}}\|{\color{black} {Q}}^{T} \Delta A {\color{black} {X}}^{-1}D_{n}\|_{F} +  \sqrt{2}\|{\color{black} {Q}} \Delta A {\color{black} {X}}^{-1}\|_{F}\|{\color{black} {Q}} \Delta A {\color{black} {X}}^{-1}D_{n}\|_{F}}{\sqrt{2} - 1},
\end{align*}
{\rm with the help of} $\|{\color{black} {Q}}^{T} \Delta A {\color{black} {X}}^{-1}D_{n}\|_{F} \leq \||{\color{black} {X}}||{\color{black} {X}}^{-1}D_{n}|\|_{2} \||{\color{black} {Q}}^{T}| K |{{\color{black} {Q}}}|\|_{F} \epsilon$ {\rm and} $\sqrt{1 + \varsigma^{2}D_{n}} \geq 1 $, {\rm we can write}\\
\begin{align*}
 \|\Delta {\color{black} {X}}{\color{black} {X}}^{-1}D_{n}\|_{F} & \leq \frac{\sqrt{2 + 2\varsigma^{2}D_{n}}\||{\color{black} {X}}||{\color{black} {X}}^{-1}D_{n}|\|_{2} \||{\color{black} {Q}}^{T}|K |{\color{black} {X}}^{-1}|\|_{F} \epsilon + (\sqrt{3} - \sqrt{2})\||{\color{black} {X}}||{\color{black} {X}}^{-1}D_{n}|\|_{2} \||{\color{black} {Q}}^{T}| K |{{\color{black} {Q}}}|\|_{F} \epsilon}{\sqrt{2} - 1} \\
 &\leq (\sqrt{3} + \sqrt{6})\||{\color{black} {X}}||{\color{black} {X}}^{-1}D_{n}|\|_{2} \||{\color{black} {Q}}^{T}| K |{{\color{black} {Q}}}|\|_{F} \epsilon.
\end{align*}
{\rm Using} $\|\Delta {\color{black} {X}}\|_{F}   = \|\Delta {\color{black} {X}} {\color{black} {X}}^{-1} D_{n}{D_{n}}^{-1}{\color{black} {X}}\|_{F}  \leq \|\Delta {\color{black} {X}} {\color{black} {X}}^{-1} D_{n}\|_{F}\|{D}^{-1}_{n}{{\color{black} {X}}}\|_{2},$
{\rm we obtain \eqref{2.41}.

Next, we prove \eqref{2.42}. For this, we use \eqref{2.34}}
\begin{align*}
{\color{black} {Q}}^{T} \Delta {\color{black} {Q}} & = {\color{black} {Q}}^{T} \Delta A {\color{black} {X}}^{-1} - \Delta {\color{black} {X}}{\color{black} {X}}^{-1} - {\color{black} {Q}}^{T} \Delta {\color{black} {Q}}\Delta {\color{black} {X}}{\color{black} {X}}^{-1},
\end{align*}
{\rm with the help of the equation \eqref{2.43},}
\begin{align*}
{\color{black} {Q}}^{T} \Delta {\color{black} {Q}} = &{\color{black} {Q}}^{T} \Delta A {\color{black} {X}}^{-1} - {\rm{\color{black}{upx}}}({\color{black} {Q}}^{T} \Delta A {\color{black} {X}}^{-1}) - {\rm{\color{black}{upx}}}({\color{black} {Q}}^{T} \Delta A {\color{black} {X}}^{-1})^{T}\nonumber\\
& - {\rm{\color{black}{upx}}}[{\color{black} {X}}^{-T}({\Delta A}^{T} \Delta A - {\Delta {\color{black} {X}}}^{T} \Delta {\color{black} {X}}){\color{black} {X}}^{-1}] - {\color{black} {Q}}^{T} \Delta {\color{black} {Q}}\Delta {\color{black} {X}}{\color{black} {X}}^{-1}.
\end{align*}
{\rm The above equation can be written as}
\begin{align*}
 {\color{black} {Q}}^{T} \Delta {\color{black} {Q}} = & {\rm{\color{black}{lowx}}}({\color{black} {Q}}^{T} \Delta A {\color{black} {X}}^{-1}) - ({\rm{\color{black}{lowx}}}({\color{black} {Q}}^{T} \Delta A X
 ^{-1}))^{T} - {\rm{\color{black}{upx}}}[{\color{black} {X}}^{-T}({\Delta A}^{T} \Delta A \nonumber\\
  &  - {\Delta {\color{black} {X}}}^{T} \Delta {\color{black} {X}}){\color{black} {X}}^{-1}] - {\color{black} {Q}}^{T} \Delta {\color{black} {Q}}\Delta {\color{black} {X}}{\color{black} {X}}^{-1}.
\end{align*}
{\rm Taking the Frobenius norm and utilizing \eqref{1.12}, we have}
\begin{align}\label{2.45}
\|{\color{black} {Q}}^{T} \Delta {\color{black} {Q}}\|_{F} & \leq \sqrt{2}\|{{\color{black} {Q}}}^{T}\Delta A {{\color{black} {X}}}^{-1}\|_{F} + \frac{1}{\sqrt{2}}{{\|\Delta A {\color{black} {X}}^{-1}\|}_{F}^{2}}  \nonumber\\
& + \frac{1}{\sqrt{2}}{{\|\Delta {\color{black} {X}} {\color{black} {X}}^{-1}\|}_{F}^{2}} + \|{\color{black} {Q}}^{T} \Delta {\color{black} {Q}}\|_{F}\|\Delta {\color{black} {X}}
{\color{black} {X}}^{-1}\|_{F}.
\end{align}
{\rm Now,}
\begin{align*}
& \|\Delta {\color{black} {X}} {\color{black} {X}}^{-1}\|_{F} \leq \frac{1}{\sqrt{2}} (4\|{{\color{black} {Q}}}^{T}\Delta A {\color{black} {X}}^{-1}\|_{F} + 2{\|{{\color{black} {Q}}}^{T}\Delta A {\color{black} {X}}^{-1}\|_{F}^{2}} ).
\end{align*}
{\rm By using $\|{{\color{black} {Q}}}^{T}\Delta A {{\color{black} {X}}}^{-1}\|_{F} \leq \sqrt{3/2} - 1,$ the previous inequality gives}
 \begin{align}\label{2.46}
{\|\Delta {\color{black} {X}} {\color{black} {X}}^{-1}\|_{F}^{2}}  & \leq (5 + 2\sqrt{6}){\|{\color{black} {Q}}^{T}\Delta A {\color{black} {X}}^{-1}\|_{F}^{2}}.
 \end{align}
{\rm By \eqref{2.ab}, \eqref{2.46}, \eqref{2.39}, and \eqref{2.45},  we have}
 \begin{align} \label{2.47}
\|{\color{black} {Q}}^{T} \Delta {\color{black} {Q}}\|_{F}  \leq & \frac{\sqrt{2}\|{{\color{black} {Q}}}^{T}\Delta A {{\color{black} {X}}}^{-1}\|_{F} + (3\sqrt{2} + 2\sqrt{3}){\|{{\color{black} {Q}}}^{T}\Delta A {\color{black} {X}}^{-1}\|_{F}^{2}}}{1 -\|\Delta {\color{black} {X}} X
^{-1}\|_{F} } \nonumber \\
\leq & \frac{\Big(\sqrt{2} + (3\sqrt{2} + 2\sqrt{3})\|{\color{black} {Q}}^{T} \Delta {\color{black} {A}} X^{-1}\|_{F}\Big)\|{\color{black} {Q}}^{T} \Delta {\color{black} {A}} X^{-1}\|_{F}}{1 - \frac{1}{\sqrt{2}}} \nonumber \\
 \leq & \frac{2 + \sqrt{6}}{\sqrt{2 - 1}}\|{\color{black} {Q}}^{T} \Delta {\color{black} {A}} X^{-1} \|_{F}.
\end{align}
{\rm Substituting \eqref{2.47} and \eqref{2.46} into \eqref{2.45}, we obtain}
\begin{align}\label{2.48}
\|{\color{black} {Q}}^{T} \Delta {\color{black} {Q}}\|_{F} \leq &  \|{{\color{black} {Q}}}^{T} \Delta A {{\color{black} {X}}}^{-1}\|_{F} + (3\sqrt{2} + 2\sqrt{3}){\|{\color{black} {Q}}^{T} \Delta {\color{black} {Q}}\|_{F}^{2}}+ \frac{(2 + \sqrt{6})(\sqrt{3} + \sqrt{2})}{\sqrt{2} - 1} {\|{\color{black} {Q}}^{T} \Delta {\color{black} {Q}}\|_{F}}.
\end{align}
{\rm By \eqref{2.38}, we have}
\begin{align}\label{2.49}
\|{\color{black} {Q}}^{T} {\Delta A} {\color{black} {X}}^{-1}\|_{F} \leq \||{\color{black} {Q}}^{T} {\Delta A} {\color{black} {X}}^{-1}|\|_{F} \leq \||{\color{black} {Q}}^{T}|K|{\color{black} {Q}}||{\color{black} {X}}||{\color{black} {X}}^{-1}|\|_{F}\epsilon.
\end{align}
{\rm Using $\|{\color{black} {Q}}^{T}\Delta A {\color{black} {X}}^{-1}\|_{F} \leq \sqrt{3/2} - 1$ and \eqref{2.49}. Then \eqref{2.48} becomes}
\begin{align*}
\|{\color{black} {Q}}^{T} \Delta {\color{black} {Q}}\|_{F} \leq(2\sqrt{2} + 2\sqrt{3} + 2 + \sqrt{6})\||{\color{black} {Q}}^{T} \Delta {\color{black} {Q}}|\|_{F} \leq \||{\color{black} {Q}}^{T}|K|{\color{black} {Q}}||{\color{black} {X}}||{\color{black} {X}}^{-1}|\|_{F}\epsilon.
\end{align*}
{\rm Since $\|\Delta {\color{black} {Q}}\|_{F} = \|{\color{black} {Q}}{\color{black} {Q}}^{T} \Delta {\color{black} {Q}}\|_{F} \leq \|{\color{black} {Q}}\|_{2}\|{\color{black} {Q}}^{T}\Delta {\color{black} {Q}}\|_{F},$ we get \eqref{2.42}.}\\
\end{proof}
{\rm Here, we prove strong rigorous perturbation bounds for factor ${\color{black}{X}}$ of the ${\color{black}{QX}}$ factorization when the perturbation $\Delta A$ takes the form of a backward error caused by the usual ${\color{black}{QX}}-$algorithm, i.e., $\Delta A \in {\mathbb{R}^{m\times n}}$. Additionally, we assume that the condition in equation \eqref{2.38} holds in this case.}
\begin{align}\label{2.50}
 \|\Psi(Z, \Delta A)\|_{F} = & \||G_{{\color{black} {X}}}|{\color{black}{ \rm vec}}(K|{\color{black} {Q}}||{\color{black} {X}}|)\|_{F}\epsilon + \||H_{{\color{black} {X}}}|{\color{black}{ \rm vec}}(|{\color{black} {X}}^{T}||{\color{black} {Q}}^{T}|K^{T}K|{\color{black} {Q}}||{\color{black} {X}}|)\|_{F}\epsilon^{2} + \||H_{{\color{black} {X}}}|\|_{2}\rho^{2} \nonumber \\
\leq & \||G_{{\color{black} {X}}}||{\color{black} {X}}^{T} \otimes I_{m}|\|_{2} \|K|{\color{black} {Q}}|\|_{F}\epsilon + \||H_{{\color{black} {X}}}||{\color{black} {X}}^{T}|\otimes |{\color{black} {X}}^{T}|\|_{2} \||{\color{black} {Q}}^{T}|K^{T}K|{\color{black} {Q}}|\|_{F}\epsilon^{2}+ \||H_{{\color{black} {X}}}|\|_{2}\rho^{2}.
\end{align}
{\rm By \eqref{2.28}, there exists an operator equation \eqref{2.29} associated with the Lyapunov majorant function, $ h(\rho, \epsilon) = \hat{a}\epsilon + \hat{b}\epsilon^{2} + \hat{c}\rho^{2},$ where $\hat{a} = \||G_{{\color{black} {X}}}||{\color{black} {X}}^{T} \otimes I_{m}|\|_{2} \|K|{\color{black} {Q}}|\|_{F}, \quad
\hat{b} = \||H_{{\color{black} {X}}}||{\color{black} {X}}^{T}|\otimes |{\color{black} {X}}^{T}|\|_{2} \||{\color{black} {Q}}^{T}|K^{T}K|{\color{black} {Q}}|\|_{F}, \\
\hat{c} = \||H_{{\color{black} {X}}}|\|_{2}.$ {\rm The Lyapunov majorant equation is given as} $ h(\rho, \epsilon) = \rho$, i.e.,  $\hat{a}\epsilon + \hat{b}\epsilon^{2} + \hat{c}\rho^{2} = \rho.$ {\rm Following the same reasoning as in Theorem \ref{thm2.2}, we find that $\epsilon \in \Omega_{2}$, where}
$\Omega_{2} = \{\epsilon \geq 0 : 1 - 4\hat{c}(\hat{a}\epsilon + \hat{b}\epsilon^{2}) > 0\}.$ The matrix equation \eqref{2.23} admits a solution that lies within the set $\{\mathfrak{B}(\epsilon) = {Z \in \mathbb{U}_{n} : \|Z\|_{F} \leq f_{1}(\epsilon)} \subset \mathbb{R}^{n \times n}\},$ {\rm where $f_{1}(\epsilon)$ is defined as} $f_{1}(\epsilon) := \frac{2(\hat{a}\epsilon + \hat{b}\epsilon^{2})}{1 + \sqrt{1 - 4\hat{c}(\hat{a}\epsilon + \hat{b}\epsilon^{2})}}.$ {\rm Consequently, the inequality $\|\Delta {\color{black} {X}}\|_{F} \leq f_{1}(\epsilon)$ holds for all $\epsilon \in \Omega_{2}$. These results lead to another significant theorem.}}

 \begin{theorem}\label{thm2.4}
Consider the unique ${\color {black}{QX}}$ factorization of $A \in {\mathbb{R}_{n}^{m \times n}}$ as in {\eqref{1.2}}, and $\Delta A \in \mathbb{R}^{m \times n}$ be a perturbation matrix in $A$ such that \eqref{2.38} holds. If
\begin{align}\label{2.51}
\hat{c}(\hat{a}\epsilon + \hat{b}\epsilon^{2}) \le \frac{1}{4},
\end{align}
then $A + \Delta A$ has a unique ${\color {black}{QX}}$ factorization \eqref{2.23} and
\begin{align}
\|\Delta {\color{black} {X}}\|_{F} \leq & \frac{2(\hat{a}\epsilon + \hat{b}\epsilon^{2})}{1 + \sqrt{1 - 4\hat{c}(\hat{a}\epsilon + \hat{b}\epsilon^{2})}}, \label{2.52} \\
\leq & 2\||G_{{\color{black} {X}}}||{\color{black} {X}}^{T} \otimes I_{m}|\|_{2} \|K|{\color{black} {Q}}|\|_{F}\epsilon + 2\||H_{{\color{black} {X}}}||{\color{black} {X}}^{T}|\otimes |X^{T}|\|_{2} \||{\color{black} {Q}}^{T}|K^{T}K|{\color{black} {Q}}|\|_{F}\epsilon^{2},   \label{2.53} \\
\leq  & (\||{\color{black} {X}}|\|_{2} + 2||G_{{\color{black} {X}}}||{\color{black} {X}}^{T} \otimes I_{m}|\|_{2})\|K|{\color{black} {Q}}|\|_{F}\epsilon. \label{2.54}
\end{align}
\end{theorem}
\begin{proof}
{\rm Clearly, we only need to show that the bound \eqref{2.54} holds. We notice the fact}
\begin{align}\label{2.55}
0 \leq 2 \hat{b}\epsilon \le \sqrt{\frac{\hat{b}}{\hat{c}}  + \hat{a}^{2}} - \hat{a} \leq \left(\frac{\hat{b}}{\hat{c}}\right)^{\frac{1}{2}} \leq \||{\color{black} {X}}|\|_{2} \|K|{\color{black} {Q}}|\|_{F},
\end{align}
{\rm which can be derived from \eqref{2.51} and \eqref{1.3}.}
\end{proof}
\begin{remark}
{\rm Using \eqref{2.55}, the condition \eqref{2.51} can be simplified to}
\begin{align}\label{2.56}
 \|H_{{\color{black} {X}}}\|_{2}(\||{\color{black} {X}}|\|_{2}\|K|{\color{black} {Q}}|\|_{F} + 2\||G_{{\color{black} {X}}}|{\color{black} {X}}^{T} \otimes I_{m}|\|_{2}\|K|{\color{black} {Q}}|\|_{F})\epsilon \le \frac{1}{2}.
 \end{align}
 \end{remark}
 \begin{remark}
 {\rm From \eqref{2.51}, we have the first-order perturbation bound as given below}
\begin{align}\label{2.57}
\|\Delta {\color{black} {X}}\|_{F} \leq \||G_{{\color{black} {X}}}||{\color{black} {X}}^{T} \otimes I_{m}|\|_{2} \|K|{\color{black} {Q}}|\|_{F}\epsilon + \mathbb{O}(\epsilon^{2}).
\end{align}
\end{remark}
\begin{remark}
{\rm The following rigorous perturbation bounds can be calculated by combining the classic and refined matrix equation approach in} \cite{chang1997pertubation,li2014new},
 \begin{align}\label{2.58}
 \|\Delta {\color{black} {X}}\|_{F} & \leq (\sqrt{6} + \sqrt{3}) (inf_{D_n \in \mathbb{D}_{n}}\sqrt{1 + \varsigma^{2}D_{n}} \|{D^{-1}}_{n}{\color{black} {X}}\|_{2} \||{\color{black} {X}}||{\color{black} {X}}^{-1}|D_{n}\|_{2} ) \|K|{\color{black} {Q}}|\|_{F}\epsilon,
 \end{align}
{\rm under the condition}
 \begin{align}\label{2.59}
  \||{\color{black} {X}}||{\color{black} {X}}^{-1}|\|_{2} \|K|{\color{black} {Q}}|\|_{F}\epsilon \le \sqrt{3/2} - 1.
\end{align}
\end{remark}
{\rm We know that the bound \eqref{2.58} can be much tighter than \eqref{2.2}. It is clear that $\||{\color{black} {Q}}^{T}|K|{\color{black} {Q}}|\|_{F} \leq \||{\color{black} {Q}}^{T}|\|_{2} \|K|{\color{black} {Q}}|\|_{F}$}.


\section{Condition numbers}\label{sec.3}
We need to consider two mappings from Definition \ref{dfn2.2} and Lemma \ref{lem2.1} to derive the explicit expressions for the condition numbers, i.e.,
\begin{align}\label{3.1}
& \varphi_1: a ={\rm {\color{black}{vec}}}(A) \rightarrow {\rm {\color{black}{ \rm xvec}}}({\color{black} {X}}), \nonumber\\
& \varphi_2: a ={\rm {\color{black}{vec}}}(A) \rightarrow {\rm {\color{black}{vec}}}({\color{black} {Q}}),
\end{align}
and using the following first-order approximations
 \begin{align*}
{\rm {\color{black}{ \rm xvec}}}(\Delta {\color{black} {X}}) = G_{{\color{black} {X}}} {\rm {\color{black}{vec}}}(\Delta A) + \text{h.o.t.}, \quad
{\rm {\color{black}{vec}}}(\Delta {\color{black} {Q}})  = G_{{\color{black} {Q}}}{\rm {\color{black}{vec}}}(\Delta A) + \text{h.o.t.},
\end{align*}
which are from \eqref{2.28} and \eqref{2.35}, respectively.

\subsection{Mixed and component-wise condition numbers}
It is well-known that norm-wise condition numbers fail to adequately capture the effects of scaling or sparsity in data. When data is improperly scaled or contains numerous zero entries, norm-wise condition numbers may not accurately reflect the magnitude of errors in these smaller or zero components. To address these issues, we explore mixed and component-wise condition numbers. Considering the general case of equation \eqref{3.1}, the definitions of mixed and component-wise condition numbers for the ${\color{black}{QX}}$ factorization are as follows:
\begin{align*}
& m_{{\color{black} {X}}}(\varphi_{{\color{black} {X}}},x)=\lim_{\epsilon \rightarrow 0}\sup_{\substack{|\Delta A|\leq \epsilon|A|}}\frac{\|\Delta {\color{black} {X}}{\color{black} {X}}\|_{\max}}{d(x+\Delta x,x)\|{\color{black} {X}}\|_{\max}}=m_{{\color{black} {X}}}, \\
&  m_{{\color{black} {Q}}}(\varphi_{{\color{black} {Q}}},x)=\lim_{\epsilon\rightarrow 0}\sup_{\substack{|\Delta A|\leq \epsilon|A|}}\frac{\|\Delta {\color{black} {Q}}\|_{\max}}{d(x+\Delta x,x)\|{\color{black} {Q}}\|_{\max}}=m_{{\color{black} {Q}}}, \\
& c_{{\color{black} {X}}}(\varphi_{{\color{black} {X}}},x)=\lim_{\epsilon\rightarrow 0}\sup_{\substack{|\Delta A|\leq \epsilon|A|}}\frac{1}{d(x+\Delta x,x)}\|\Delta {\color{black} {X}}/{\color{black} {X}}\|_{\max}=c_{{\color{black} {X}}}, \\
& c_{{\color{black} {Q}}}(\varphi_{{\color{black} {Q}}},x)=\lim_{\epsilon\rightarrow 0}\sup_{\substack{|\Delta A|\leq \epsilon|A|}}\frac{1}{d(x+\Delta x,x)}\|\Delta {\color{black} {Q}}/{\color{black} {Q}}\|_{\max}=c_{{\color{black} {Q}}}.
\end{align*}
where $ x = {\rm {\color{black}{vec}}}(A),$  and for a matrix $A$, $\|A\|_{\max}=\|{\rm {\color{black}{vec}}}(A)\|_{\infty}=\mathop {\max }\limits_{i,j} |a_{ij}|.$ Taking into account Lemma \ref{lem2.1} and the preceding equations, the following theorem provides the explicit expressions for the mixed and component-wise condition numbers for both factors in the ${\color {black}{QX}}$ factorization.
\begin{theorem}\label{thm3.2}
Assume that the conditions in Theorem \ref{thm2.1} hold. Then, the mixed and component-wise condition numbers for the factors ${\color{black} {X}}$ and ${\color{black} {Q}}$ are given by:
\begin{align}
  m_{{\color{black} {X}}} = \frac{\||G_{{\color{black} {X}}}|{\rm {\color{black}{vec}}}(|A|)\|_{\infty}}{\|{\color{black} {X}}\|_{\max}}, \quad\quad\quad &  c_{{\color{black} {X}}}= { \left\| \frac{|G_{{\color{black} {X}}}|{\rm {\color{black}{vec}}}(|A|)}{|{\color{black} {X}}|} \right\|}_{\infty}, \label{3.2} \\
  m_{{\color{black} {Q}}} = \frac{\||G_{{\color{black} {Q}}}|{\rm {\color{black}{vec}}}(|{\color{black} {Q}}|)\|_{\infty}}{\|{\color{black} {Q}}\|_{\max}},  \quad\quad\quad &  c_{{\color{black} {Q}}}={ \left\| \frac{|G_{{\color{black} {Q}}}|{\rm {\color{black}{vec}}}(|A|)}{|{\color{black} {Q}}|} \right\|}_{\infty}. \label{3.3}
\end{align}
\end{theorem}
Below are the upper bounds for the mixed and component-wise condition numbers, as the matrices in \eqref{3.2} and \eqref{3.3} are quite large and computationally expensive to compute.
\begin{corollary}
Using the same hypotheses as in Theorem \ref{thm3.2}, we have
\begin{align}
& m_{{\color{black} {X}}} \leq \frac{\|{\rm {\color{black}{upx}}} ((|{\color{black} {X}}^{-T}||A^{T}||{\color{black} {Q}}|) + (|{\color{black} {Q}} ^{T}||A||{\color{black} {X}}^{-1}|))|{\color{black} {X}}|\|_{\max}}{\|{\color{black} {X}}\|_{\max}} = {m_{{\color{black} {X}}}^{upp}}(A),   \label{3.4}\\
& c_{{\color{black} {X}}} \leq  \frac{\|{\rm {\color{black}{upx}}} ((|{\color{black} {X}}^{-T}||A^{T}||{\color{black} {Q}}|) (|{\color{black} {Q}}^{T}||A||{\color{black} {X}}^{-1}|))|{\color{black} {X}}|\|_{\max}}{\|{\color{black} {X}}\|_{\max}} = {c_{{\color{black} {X}}}^{upp}}(A),      \label{3.5}\\
& m_{{\color{black} {Q}}} \leq \frac{\||{\color{black} {X}}^{-1}||A| + |{\color{black} {Q}}|{\rm {\color{black}{upx}}} ((|{\color{black} {X}}^{-T}||A^{T}||{\color{black} {Q}}|) + (|{\color{black} {Q}}^{T}||A||{\color{black} {X}}^{-1}|))\|_{\max}}{\|{\color{black} {Q}}\|_{\max}} = {m_{{\color{black} {Q}}}^{upp}}(A), \label{3.6}
\end{align}
and
\begin{align} & c_{{\color{black} {Q}}} \leq \left\| \frac{|{\color{black} {X}}^{-1}||A| + |{\color{black} {Q}}|{\rm {\color{black}{upx}}} ((|{\color{black} {X}}^{-T}||A^{T}||{\color{black} {Q}}|) + (|{\color{black} {Q}}^{T}||A||{\color{black} {X}}^{-1}|))}{|{\color{black} {Q}}|}  \right\|_{\max} = {c_{{\color{black} {Q}}}^{upp}}(A). \label{3.7}
\end{align}
\end{corollary}
\begin{proof}
 The bound \eqref{3.4} can be derived from \eqref{3.2}, we therefore have
  \begin{align*}
& \||M_{\rm {\color{black}{ \rm xvec}}}(X
^{T}\otimes I_{n})M_{{\rm {\color{black}{upx}}}}[({\color{black} {X}}^{-T}\otimes I_{n}) + (I_{n}\otimes {\color{black} {X}}^{-T} )\Pi_{m,n}]| {\rm {\color{black}{vec}}}(|A|)\|_{\infty} \\
& \leq \|M_{\rm {\color{black}{ \rm xvec}}}(|X
^{T}|\otimes |I_{n}|)M_{{\rm {\color{black}{upx}}}}[(|{\color{black} {X}}^{-T}|\otimes |I_{n}|) + (|I_{n}|\otimes |{\color{black} {X}}^{-T}| )\Pi_{m,n}]| {\rm {\color{black}{vec}}}(|A|)\|_{\infty} \\
& = \|{\rm {\color{black}{ \rm xvec}}}(|I_{n}|{\rm {\color{black}{upx}}}((|{\color{black} {X}}^{-T}||A^{T}||{\color{black} {Q}}|) + (|{\color{black} {Q}}^{T}||A||{\color{black} {X}}^{-1}|))|{\color{black} {X}}|)\|_{\infty} \\
& \leq \|{\rm {\color{black}{upx}}}((|{\color{black} {X}}^{-T}||A^{T}||{\color{black} {Q}}|) + (|{\color{black} {Q}}^{T}||A||{\color{black} {X}}^{-1}|))|{\color{black} {X}}|\|_{\max}.
 \end{align*}
Likewise, we can acquire \eqref{3.5}, \eqref{3.6}, and \eqref{3.7}.
\end{proof}

\section{Numerical Experiments}\label{sec.4}
In this section, we provide numerical examples to illustrate the results derived in previous sections. All computations are carried out in MATLAB, R2023a.
\begin{example}\label{ex1}
{\rm When \( m > n \), we first generate a centrosymmetric matrix \( A \in R^{m \times n} \) using the MATLAB function \texttt{centro\_generator.m}, which is derived from \texttt{centrocell.m} and can be downloaded from the Bitbucket repository: \href{https://bitbucket.org/ajsteele/centrosymmetric-qx/downloads/}{https://bitbucket.org/ajsteele/centrosymmetric-qx/downloads/}, as described in \cite{lv2022perturbation}. For the scaling matrix in \eqref{2.2}, as done in \cite{lv2022perturbation}, and also for \eqref{2.41}, we choose \( D_{n} = D_{r} \), such that \( D_{r} = {\rm{diag}}\|{\color{black} {X}}(i,:)\|_{2} \). For the scaling matrix in \eqref{2.3} and \eqref{2.42}, we  consider \( D_{n} = I_{n} \). To compute the rigorous perturbation bounds of the factor \( {\color{black} {X}} \), we use \eqref{2.2}, \eqref{2.22}, \eqref{2.41}, and \eqref{2.54}, and for the factor \( {\color{black} {Q}} \), we use \eqref{2.3} and \eqref{2.42}. The numerical tests for \( m \ge n \) are shown in Table \ref{table1}, and the same values are shown in Table \ref{table2} and Table \ref{table3}. From Tables \ref{table1}, \ref{table2}, and \ref{table3}, we can quickly observe that the bounds obtained using the modified matrix equation approach for the factor \( {\color{black} {X}} \) are always sharper than the bounds obtained using the refined matrix equation approach when the original matrix is perturbed norm-wise or component-wise. For the factor \( {\color{black} {Q}} \), we can see from the tables that the bounds obtained using the refined matrix equation approach are sharper than those obtained using the modified matrix equation approach.
}
\end{example}
For simplicity, we use the following notation:
\begin{align*}
& b_{\Delta {\color{black} {X}}1} = (\||{\color{black} {X}}|\|_{2} + 2||G_{{\color{black} {X}}}|{\color{black} {X}}^{T} \otimes I_{m}||\|_{2})\|K|{\color{black} {Q}}|\|_{F}\epsilon, \\
& b_{\Delta {\color{black} {X}}2} = (\sqrt{3} + \sqrt{6}) \left( \mathop {\inf}\limits_{D_{n} \in \mathbb{D}_n}\||{\color{black} {X}}||{\color{black} {X}}^{-1}|D_{n}\|_{2} \|{{D}_{n}}^{-1}{\color{black} {X}}\|_{2}\right)\||{\color{black} {Q}}^{T}| K|{\color{black} {Q}}|\|_{F}\epsilon,\\
& b_{\Delta {\color{black} {X}}3} =(1 + 2\|G_{{\color{black} {X}}}\|_{2}), \\
& b_{\Delta {\color{black} {X}}4} = (\sqrt{6} + \sqrt{3}) \left(\mathop {\inf}\limits_{D_{n} \in \mathbb{D}_n} \sqrt{1 + {\varsigma^{2}}_{D_{n}}}\kappa_{2} (D^{-1}{\color{black} {X}})\right) \|{\color{black} {Q}}\|_{2}\|\Delta A\|_{F}, \\
& b_{\Delta {\color{black} {Q}}1} =  (\sqrt{6} + 2 + 2\sqrt{2} + 2\sqrt{3}) \||{\color{black} {Q}}^{T}| K|{\color{black} {Q}}|\|_{F}cond({\color{black} {X}}) \epsilon, \\
& b_{\Delta {\color{black} {Q}}2} = \left( 2 + \sqrt{2} \right)\left[ \|G_{{\color{black} {Q}}}\|_{2} + \|{\color{black} {X}}^{-T} \otimes {\color{black} {Q}}\|_{2} (1 + \|G_{{\color{black} {X}}}\|_{2}) \right]\|\Delta A\|_{F},\\
& b_{\Delta {\color{black} {Q}}3} = (2\sqrt{2} + 2) \left ( \mathop {\inf}\limits_{D_{n} \in \mathbb{D}_n} \|{\color{black} {Q}}{D^{-1}}_{n}\|_{2}\|{\color{black} {X}}^{-1}D_{n}\|_{2}\right) \|{\color{black} {Q}}\|_{2} \|\Delta A\|_{F} + (2\sqrt{3} + \sqrt{6})\|{\color{black} {Q}}^{T} \Delta A {\color{black} {X}}^{-1}\|_{F}.
\end{align*}
\begin{table}[H]
	
	\caption[Table : 1]{ Rigorous perturbation bounds when \textit{m} $>$ \textit{n}}
	\centering
\scriptsize
\footnotesize\setlength{\tabcolsep}{4pt}
	\label{table1}
\scalebox{0.7}{
	\begin{tabular}{c c c c c c c c c c c c c}
		\hline
		$m, n$ & $\|\Delta A\|_{F}$ & $\|\Delta {\color{black} {X}}\|_{F}$ & $\|\Delta {\color{black} {Q}}\|_{F}$ & $\|{\color{black} {Q}}^{T}\Delta {\color{black} {Q}}\|_{F}$ &{\rm cond}(A)& $b_{\Delta {\color{black} {X}}1}$ &$b_{\Delta {\color{black} {X}}2}$& $b_{\Delta {\color{black} {X}}3}$&$b_{\Delta {\color{black} {X}}4}$ &$b_{\Delta {\color{black} {Q}}1}$ &$b_{\Delta {\color{black} {Q}}2}$ & $b_{\Delta {\color{black} {Q}}3}$ \\
		\hline\hline
20,10 &7.9925e-07 &8.3509e-07 &4.6451e-07 &2.4495e-07 &23.9970 &1.0294e+03 &2.4396e+03 &12.4659	&84.8637 &1.7351e+03 &76.4609 &30.9115\\
30,20 &1.3556e-07 &1.4565e-07 &6.6223e-08 &5.2237e-08 &31.7693 &6.8193e+03 &1.2624e+04 &12.6827 &96.8450 &8.2824e+03 &72.7588 &28.8308\\
40,20 &1.6137e-08	&1.5706e-08	&6.0011e-09	&3.7952e-09	&20.7983 &9.5586e+03 &1.4578e+04 &10.3463 &71.4964 &8.7992e+03 &33.5324 &15.5300\\
50,30 &2.2429e-09	&2.3070e-09	&7.4221e-10	&4.7902e-10	&53.8789 &3.7532e+04 &5.9398e+04 &12.8052 &107.9084 &3.4753e+04 &72.6922 &28.8775\\
60,30 &2.5060e-10	&2.5264e-10	&7.4927e-11	&4.9649e-11	&25.9020 &3.8596e+04 &5.9268e+04 &14.4068 &126.4451 &2.4300e+04 &37.2434 &13.2672\\
70,40 &3.0597e-11	&3.1500e-11	&8.7768e-12	&6.1222e-12	&37.8395 &1.0816e+05 &1.5975e+05 &16.2680 &157.6286 &6.2049e+04 &46.7614 &15.1576\\
150,50&4.9729e-12	&4.8168e-12	&8.5410e-13	&4.6336e-13	&27.1203 &4.7770e+05 &5.7888e+05 &15.8903 &167.2820 &1.3110e+05 &20.3381 &6.6799\\
200,60 &6.3359e-13	&6.1617e-13	&9.2144e-14	&4.8274e-14	&31.3745 &1.0840e+06 &1.2799e+06 &18.3472 &191.6962 &2.2883e+05 &20.8784 &6.0735\\
300,100 &9.9996e-14	&9.2731e-14	&1.4508e-14	&7.8925e-15	&39.4837 &6.5352e+06 &6.8403e+06 &23.2079 &270.1693 &9.4562e+05 &20.4557 &4.8572\\

		\hline
	\end{tabular}
}
\end{table}

\begin{table}[H]
	
	\caption[Table:2 ]{Rigorous perturbation bounds \textit{m} = \textit{n}}
	\centering
\scriptsize
\footnotesize\setlength{\tabcolsep}{4pt}
	\label{table2}
\scalebox{0.7}{
	\begin{tabular}{c c c c c c c c c c c c}
		\hline
$m=n$ & $\|\Delta A\|_{F}$ & $\|\Delta {\color{black} {X}}\|_{F}$ & $\|\Delta {\color{black} {Q}}\|_{F}$ &cond(A)& $b_{\Delta X
1}$ &$b_{\Delta {\color{black} {X}}2}$& $b_{\Delta {\color{black} {X}}3}$ & $b_{\Delta {\color{black} {X}}4}$&$b_{\Delta {\color{black} {Q}}1}$ &$b_{\Delta {\color{black} {Q}}2}$ &$b_{\Delta {\color{black} {Q}}3}$\\
		\hline\hline

10 &5.9910e-08	&6.472e-08	&3.3599e-08	&45.8302	&280.3398	&522.5655	&7.9049 	&49.1338	&1.6848e+03	&162.4038	&96.3765\\
10 &5.8900e-09	&7.1522e-09	&1.5145e-08	&87.5321	&372.2899	&1.0665e+03	&19.7034	&149.1569	&2.2712e+03	&741.1668	&200.0122\\
20 &1.1634e-09	&1.4375e-09	&1.7169e-09	&305.7068	&7.8718e+03	&2.0687e+04	&34.9743	&200.2455	&4.8283e+04	&1.9074e+03	&316.2860\\
20 &1.0920e-10	&1.3094e-10	&1.0013e-10	&309.0977	&7.2827e+03	&2.2529e+04	&38.0592	&274.6799	&3.8752e+04	&2.0729e+03	&315.3371\\
30 &1.7211e-10	&1.9684e-10	&1.1781e-10	&738.6378	&2.2388e+04	&5.7245e+04	&22.9532	&182.9586	&2.7194e+05	&2.2332e+03	&542.5035\\
30 &1.7116e-11	&1.9622e-11	&9.2663e-12	&243.8271	&1.8877e+04	&4.2571e+04	&14.6777	&163.3906	&7.0333e+06	&487.1888	&171.5523\\
100 &5.7964e-12	&6.7628e-12	&2.2180e-12	&3.2634e+03	&2.2522e+06	&5.6042e+06	&31.4992	&491.5955	&1.0288e+07	&3.7650e+03	&697.5794\\
110 &6.3472e-13	&7.3039e-13	&2.9028e-13	&1.3770e+03	&3.1235e+06	&6.8006e+06	&29.0173	&342.0829	&7.5001e+06	&1.3805e+03	&265.8052\\
120 &6.9589e-14	&1.0310e-13	&3.9503e-14	&2.9060e+03	&5.1020e+06	&1.6126e+07	&62.9860	&1.0971e+03	&1.5415e+07	&5.4084e+03	&513.7461\\

		\hline
	\end{tabular}
}
\end{table}

{\rm Presently, the MATLAB functions $\textbf{toeplitz(s)}$, and $s = \textbf{randn(n,1)}$, will be used to generate a centrosymmetric matrix $A \in {\mathbb{R}_{n}}^{m\times n}$. As mentioned earlier, numerical results comparing the perturbation bounds for the ${\color{black}{QX}}$ factorization can be obtained, and these are also presented in Table \ref{table3}.}

\begin{table}[H]	

	\caption[Table:3 ]{Rigorous perturbation bounds for the case \textit{m} = \textit{n}}
	\centering
\scriptsize
\footnotesize\setlength{\tabcolsep}{4pt}
	\label{table3}
\scalebox{0.7}{
	\begin{tabular}{c c c c c c c c c c c c c }
		\hline
		$n$& $\|\Delta A\|_{F}$ & $\|\Delta {\color{black} {X}}\|_{F}$ & $\|\Delta {\color{black} {Q}}\|_{F}$ & cond(A)& $b_{\Delta {\color{black} {X}}1}$ &$b_{\Delta X
2}$& $b_{\Delta {\color{black} {X}}3}$ & $b_{\Delta {\color{black} {X}}4}$&$b_{\Delta {\color{black} {Q}}1}$ &$b_{\Delta {\color{black} {Q}}2}$ &$b_{\Delta {\color{black} {Q}}3}$\\
		\hline\hline
10 &8.9523e-08	&1.211e-07	&1.0172e-07	&15.5083	&308.5180	&580.0241	&9.4198	&45.6652	 &604.4170	 &69.0286	&35.2370\\
10 &1.1185e-08	&1.2780e-08	&5.6063e-09	&14.8278	&258.5796	&495.6520	&7.9020	&48.9284	 &579.1363	 &57.3936	&33.2366\\
20 &1.6249e-09	&2.9381e-09	&1.6610e-09	&37.5039	&5.2980e+03	&1.1127e+04	&21.3719	&97.3340	 &1.1487e+04	 &260.3394	&65.0442\\
20 &1.5612e-10	&2.6547e-10	&1.1922e-10	&42.5995	&6.1135e+03	&9.8811e+03	&33.8789	&143.4470  &7.0400e+03	&312.5556	&52.2861\\
30 &2.7758e-10	&4.3753e-10	&1.1513e-10	&13.9552	&1.6209e+04	&2.0508e+04	&11.4808	&49.9367	  &9.8642e+03	&41.5541	&17.0939\\
30 &3.1889e-11	&4.3933e-11	&9.0485e-12	&26.8157	&1.9790e+04	&3.1505e+04	&10.3940	&54.6332	  &1.6053e+04	&47.8899	&22.6328\\
100 &9.3497e-12	&1.5348e-11	&2.9394e-12	&108.4877	&1.9230e+06	&3.0991e+06	&25.9406	&229.9995	&1.0345e+06	&233.0123	&49.2660\\
110 &1.1719e-12	&2.0085e-12	&2.7850e-13	&61.5498	&3.4262e+06	&4.2386e+06	&22.8658	&152.0776	&8.7795e+05	&97.4466	&23.0451\\
120 &1.1547e-13	&2.1338e-13	&7.3607e-14	&272.2991	&3.5249e+06	&5.6141e+06	&30.5836	&208.1417	&4.2193e+06	&776.3241	&147.4020\\

		\hline
	\end{tabular}
}
\end{table}

\begin{example}
{\rm We used the test matrices from Example \ref{ex1} to calculate condition numbers and their upper bounds. The numerical results are shown in Tables \ref{table4} and \ref{table5}. As a result, we used \(k_{{\color{black} {X}}}(A)^{\text{upper}}\) and \(k_{{\color{black} {Q}}}(A)^{\text{upper}}\) to calculate the maximum possible magnification of the error for the factors \({\color{black} {X}}\) and \({\color{black} {Q}}\), respectively, with respect to a minor change in \(A\), because the upper bounds of the condition numbers for the factors \({\color{black} {X}}\) and \({\color{black} {Q}}\) are always good estimates.
}
\end{example}

\begin{table}[H]	

	\caption[Table :4]{ When $m > n$}
	\centering
\scriptsize
\footnotesize\setlength{\tabcolsep}{4pt}
	\label{table4}
\scalebox{0.7}{
	\begin{tabular}{c c c c c c c c c c}
		\hline	$(m,n)$&$k(A)$&$k_{{\color{black} {X}}(c)}$&$k_{{\color{black} {X}}(cu)}$&$k_{{\color{black} {X}}(m)}$&$k_{{\color{black} {X}}(mu)}$&$k_{{\color{black} {Q}}(c)}$&$k_{{\color{black} {Q}}(cu)}$&$k_{{\color{black} {Q}}(m)}$&$k_{{\color{black} {Q}}(mu)}$ \\
		\hline\hline

(20,10) &12.4617	&1.0333e+03	    &3.5178e+03	    &3.7995	&7.6828	    &2.5326e+03	    &8.4145e+03	&9.5068 	&32.8103\\
(30,20) &29.7527	&2.1002e+03	    &8.8973e+03	    &6.7063	&17.3916	&3.2962e+03	    &1.1781e+04	&23.9260	&110.9770\\
(40,20) &20.2914	&1.3383e+03	    &2.2777e+03	    &4.5276	&14.7067	&1.8541e+05	    &6.8309e+05	&16.9065	&82.8431\\
(50,30) &39.3927	&1.5001e+04	    &3.0974e+04 	&8.7242	&27.2569	&4.7307e+03	    &2.4771e+04	&40.0854	&244.0194\\
(60,30) &24.9315	&6.0510e+03 	&1.0287e+04	    &7.4345	&20.2774	&7.1959e+04 	&1.8812e+05	&30.8075	&169.5142\\
(70,40) &31.5061	&2.4272e+04	    &1.3903e+05	    &7.6969	&31.0454	&9.5864e+04 	&5.0850e+05	&33.9196	&226.0615\\
(150,50) &26.5536	&2.3546e+04	    &7.8643e+04	    &6.8743	&20.5070	&3.8757e+05 	&2.1251e+06	&40.9936	&226.6025\\
(200,60) &31.4820	&2.7964e+05	    &9.0872e+05 	&7.1037	&24.6605	&1.4521e+05	    &7.5048e+05	&45.4397	&280.74841\\
(300,100) &44.2607	&8.0027e+04	    &3.1181e+05 	&9.1393	&35.7994	&4.3992e+05 	&2.6563e+06	&75.4188	&584.8364\\

		\hline
	\end{tabular}
}
\end{table}

\begin{example}
{\rm By using
\begin{align*}
  \scalebox{0.7}{$ B = [ \begin{array}{cccccccccc}
     10^{-e} & 1/2000 & 1/3000 & 1/4000 & 1/5000 &  1/6000 & 1/7000 & 1/8000 & 1/9000 & 1/10^{-e}
   \end{array} ] $}
\end{align*}
to generate an ill-conditioned, centrosymmetric matrix, we obtain the following results}
\end{example}

\begin{table}[H]	
	\caption[Table : 5]{When \textit{ m} $>$ \textit{n}}
	\centering
\scriptsize
\footnotesize\setlength{\tabcolsep}{2pt}
	\label{table5}
\scalebox{0.7}{
	\begin{tabular}{c c c c c c c c c c c}
		\hline
	$(m,n)$&$e$&$kA$&$k_{X
c}$&$k_{{\color{black} {X}}cu}$&$k_{{\color{black} {X}}m}$&$k_{{\color{black} {X}}mu}$ &$k_{{\color{black} {Q}}c}$&$k_{{\color{black} {Q}}cu}$&$k_{{\color{black} {Q}}m}$&$k_{{\color{black} {Q}}mu}$ \\
		\hline\hline
(5,4) & 1 &7.0000e+06	&6.5278	&201.0007	&1	&1.0000	&6.4874e+06	&1.1200e+08	&2.2500 &28.0001\\
(5,4) &0 &7.0998e+03    &3.8167e+03 &9.3636e+03 &1.0001 &3.5362 &1.7273e+04 &1.8611e+07 &1.2111e+03 &1.0566e+04\\
(5,4) &-1 &2.9207e+05	&27.4049 &5.1991e+03	&1.0000	&1.0000	&2.0220e+04	&2.0499e+08	&17.9881	&6.5063e+03\\
(5,4) &-4 &3.1500+08	&14.8602	&17.0743	&1	&1.0000	&2.0407+07	&2.9684+08	&10.8834	&11.5772\\
(5,4) &-3 &3.1500e+07	&14.8602	&58.8670	&1.0000	&1.0000	&2.0081e+06	&2.1883e+08	&10.8833	&69.4630\\

  \hline
	\end{tabular}
}
\end{table}

{\rm By using
\begin{align*}
  \scalebox{0.7}{$ B = [ \begin{array}{cccccccccc}
    1000 &1/2000 &1/3*10^{-e} &1/4000 &1/5000 &1/6000 &1/7*10^{-e} &1/8000 &1/9000 &1/1000
   \end{array} ] $}
\end{align*}
to generate an ill-conditioned, centrosymmetric matrix, we obtain the following results}

\begin{table}[H]	
	\caption[Table : 6]{When \textit{ m} $>$ \textit{n}}
	\centering
\scriptsize
\footnotesize\setlength{\tabcolsep}{2pt}
	\label{table6}
\scalebox{0.7}{
	\begin{tabular}{c c c c c c c c c c c}
		\hline
	$(m,n)$&$e$&$kA$&$k_{{\color{black} {X}}c}$&$k_{{\color{black} {X}}cu}$&$k_{{\color{black} {X}}m}$&$k_{{\color{black} {X}}mu}$ &$k_{{\color{black} {Q}}c}$&$k_{Q
cu}$&$k_{{\color{black} {Q}}m}$&$k_{{\color{black} {Q}}mu}$ \\
		\hline\hline

(5,4)&-4   &4.8189	 &34.9958   &75.4366	&1.4221	&4.6545	&2.8171e+06	&2.2240e+06	&2.0238	&7.2148\\
(5,4)& 4   &1.0328e+07 &5.4705	&11.7500	&1.0000	&1.0000	&2.8870e+06	&1.0278e+08	&2.5922	&29.8836\\
(5,4)& 3   &3.1500e+07 &14.8602	&58.8670	&1.0000	&1.0000	&2.0081e+06	&2.1883e+08	&10.8833 &69.4630\\

 \hline
	\end{tabular}
}
\end{table}

and
\begin{align*}
  \scalebox{0.5}{$ B = [ \begin{array}{cccccccccccccccccc}
    10^{-e} &1/2000 &1/3000 &1/4000 &1/5000 &1/6000 &1/7000 &1/8000 &1/9000 &1/10^{-e} &1/4000 &1/5000 &1/6000 &1/7*10^{-e} &1/8000 &1/9000 &1/1000 &1/3*10^{-e}
   \end{array} ] $}
\end{align*}  for $m = n$

\begin{table}[H]	
	\caption[Table : 6]{When $m$ = $n$}
	\centering
\scriptsize
\footnotesize\setlength{\tabcolsep}{2pt}
	\label{table7}
\scalebox{0.7}{
	\begin{tabular}{c c c c c c c c c c c c c}
		\hline
	$(m,n)$&$e$&$kA$&$k_{{\color{black} {X}}c}$&$k_{{\color{black} {X}}cu}$&$k_{{\color{black} {X}}m}$&$k_{{\color{black} {X}}mu}$ &$k_{{\color{black} {Q}}c}$&$k_{{\color{black} {Q}}cu}$&$k_{{\color{black} {Q}}m}$&$k_{{\color{black} {Q}}mu}$ \\
		\hline \hline

6 & 2   &2.0006e+05	 &20.6675	&119.6456	&2.2833	&5.5034	&41.8555	&197.6872	&1.9220	&10.3245\\
6 & 3   &1.7890e+06  &109.3191	&151.2719	&2.8785	&4.5011	&77.6157 	&274.4370	&3.6545	&10.4410\\
6 & 4   &4.0475e+08  &18.8116	&365.2161	&2.2138	&5.6320	&250.0004	&1023.3962	&5.1950	&32.5212\\
6 & 5   &2.1501e+09  &21.2200	&39.8968	&4.0760	&7.9269	&230.1613	&588.5895	&4.1830	&10.9114\\

 \hline
	\end{tabular}
}
\end{table}

\section{Conclusion}\label{sec.5}

This study has thoroughly explored the rigorous perturbation analysis of the ${\color {black}{QX}}$ decomposition for centrosymmetric matrices, addressing both norm-wise and component-wise perturbations of the original matrix. Weak rigorous perturbation bounds were derived using the refined matrix equation approach, while strong rigorous perturbation bounds were obtained by combining the modified matrix-vector equation technique, the Lyapunov majorant function, and the Banach fixed-point theorem. Additionally, we provided explicit expressions for the mixed and component-wise condition numbers, along with their upper bounds. The numerical results confirm the validity of these findings. Future work will extend this analysis by exploring the multiplicative perturbation bounds for the ${\color {black}{QX}}$ decomposition.

\section*{Acknowledgment}
The second author is supported by the project TUBITAK 1001, 123F356. We are grateful to the anonymous referees and the editor, for their valuable comments and suggestions, which have significantly improved the manuscript.
\bibliographystyle{unsrt}
\bibliography{references}
\end{document}